\documentclass{article}

\RequirePackage{authblk}
\RequirePackage{xargs}
\RequirePackage[OT1]{fontenc}
\RequirePackage{amsthm,amsmath,txfonts,bm,pifont,graphicx,bbm,enumitem}
\RequirePackage[numbers]{natbib}
\RequirePackage[colorlinks,citecolor=blue,urlcolor=blue]{hyperref}
\RequirePackage[utf8]{inputenc}
\RequirePackage[ruled,vlined]{algorithm2e}
\SetKwHangingKw{KwInit}{initialization}
\SetKwHangingKw{KwIn}{input}

\numberwithin{equation}{section}
\theoremstyle{plain}
\newtheorem{theorem}{Theorem}[section]
\newtheorem{proposition}{Proposition}

\newtheorem{lemma}{Lemma}

\theoremstyle{plain}
\newtheorem{definition}{Definition}

\newcommand{\1}{\ensuremath{\mathbbm{1}}}
\newcommand{\rme}{\mathrm{e}}
\newcommand{\rmd}{\mathrm{d}}
\newcommand{\N}{\mathbb{N}}
\newcommand{\Z}{\mathbb{Z}}
\newcommand{\R}{\mathbb{R}}
\newcommand{\EE}{\mathbb{E}}

\newcommand{\K}{\mathcal{K}}
\newcommand{\btheta}{\bm\theta}

\newenvironment{hypothesis}[1]{
\begin{enumerate}[label={\sf(\textbf{#1}-\arabic*)},resume=hypothesis#1]\begin{sf}}
{\end{sf}\end{enumerate}}
\newenvironment{hypothesis*}[1]{
\begin{enumerate}[label={\sf(\textbf{#1})}]\begin{sf}}
{\end{sf}\end{enumerate}}

\makeindex             

\begin{document}

\title{Time series prediction via aggregation: an oracle bound including numerical cost}

\author{Andres Sanchez-Perez\\
\texttt{andres.sanchez-perez@telecom-paristech.fr}\\
Institut Mines-T\'el\'ecom ; T\'el\'ecom ParisTech ; CNRS LTCI}


\maketitle

\begin{abstract}
We address the problem of forecasting a time series meeting the Causal
Bernoulli Shift model, using a parametric set of predictors. The aggregation technique provides a predictor with well established and quite satisfying theoretical properties expressed by an oracle inequality for the prediction risk. The numerical computation of the aggregated predictor usually relies on a Markov chain Monte Carlo method whose convergence
should be evaluated. In particular, it is crucial to bound the number of
simulations needed to achieve a numerical precision of the same order as the
prediction risk. In this direction we  present a fairly general result which can be seen as an oracle inequality including the numerical cost of the predictor computation. The numerical cost appears by letting the oracle inequality depend on the number of simulations required in the Monte Carlo approximation. Some numerical experiments are then carried out to support
our findings.
\end{abstract}

\section{Introduction}
The objective of our work is to forecast a stationary time series $Y=\left(Y_{t}\right)_{t\in\Z}$ taking values in $\mathcal{X}\subseteq\R^{r}$ with $r\geq 1$. For this purpose we propose and study an aggregation scheme using exponential weights.

Consider a set of individual predictors giving their predictions at each moment $t$. An aggregation method consists of building a
new prediction from this set, which is nearly as good as the best among the individual ones, provided a
risk criterion (see \cite{Leung_Barron:2006}). This kind of result is established by oracle inequalities. The power and the beauty of the technique lie in its simplicity and versatility. The more basic and general context of application is individual sequences, where no assumption on the observations is made (see \cite{Cesa-Bianchi_Lugosi:2006} for a comprehensive overview). Nevertheless, results need to be adapted if we set a stochastic model on the observations. 

The use of exponential weighting in aggregation and its links with the PAC-Bayesian approach has been investigated for example in \cite{Audibert:2004}, \cite{Catoni:2004} and \cite{Dalalyan_Tsybakov:2008}. Dependent processes have not received much attention from this viewpoint, except in \cite{Alquier_Li:2012} and \cite{Alquier_Wintenberger:2012}. In the present paper we study the properties of the Gibbs predictor, applied to Causal Bernoulli Shifts (CBS). CBS are an example of dependent processes (see \cite{Dedecker_Doukhan_Lang:2007} and \cite{Dedecker_Prieur:2005}).

Our predictor is expressed as an integral since the set from which we do the aggregation is in general not finite. Large dimension is a trending setup and the computation of this integral is a major issue. We use classical Markov chain Monte Carlo (MCMC) methods to approximate it. Results from {\L}atuszy{\'n}ski
\cite{Latuszynski_Miasojedow_Niemiro:2013}, \cite{Latuszynski_Niemiro:2011} control the number of MCMC iterations to obtain precise bounds for the approximation of the integral. These bounds are in expectation and probability with respect to the distribution of the underlying Markov chain.

In this contribution we first slightly revisit certain lemmas presented in \cite{Alquier_Wintenberger:2012}, \cite{Catoni:2004} and \cite{Rio:2000} to derive an oracle bound for the prediction risk of the Gibbs predictor. We stress that the inequality controls the convergence rate of the exact predictor. Our second goal is to investigate the impact of the approximation of the predictor on the convergence guarantees described for its exact version. Combining the PAC-Bayesian bounds with the MCMC control, we then provide an oracle inequality that applies to the
MCMC approximation of the predictor, which is actually used in practice. 

The paper is organised as follows: we introduce a motivating example and several definitions and assumptions in Section~\ref{section:statement_of_the_problem_and_notation}. In Section~\ref{section:prediction_via_aggregation} we describe the methodology of aggregation and provide the oracle inequality for the exact Gibbs predictor. The stochastic approximation is studied in Section~\ref{section:stochastic_approximation}. We state a general proposition independent of the model for the Gibbs predictor. Next, we apply it to the more particular framework delineated in our paper. A concrete case study is analysed in Section~\ref{section:numerical_work}, including some numerical work. A brief discussion follows in Section~\ref{section:discussion}. The proofs of most of the results are deferred to Section~\ref{section:technical_proofs}.

Throughout the paper, for $\bm a\in\R^{q}$ with $q\in\N^{*}$, $\|\bm a\|$ denotes its Euclidean norm, $\|\bm a\|=(\sum_{i=1}^{q}a_{i}^{2})^{1/2}$ and $\|\bm a\|_{1}$ its $1$-norm  $\|\bm a\|_{1}=\sum_{i=1}^{q}|a_{i}|$. We denote, for $\bm a\in\R^{q}$ and $\Delta>0$, $B\left(\bm a,\Delta\right)=\{\bm a_{1}\in\R^{q}: \|\bm a-\bm a_{1}\|\leq\Delta\}$ and $B_{1}\left(\bm a,\Delta\right)=\{\bm a_{1}\in\R^{q}: \|\bm a-\bm a_{1}\|_{1}\leq\Delta\}$ the corresponding balls centered at $\bm a$ of radius $\Delta>0$.
In general bold characters represent column vectors and normal characters their components; for example $\bm y=\left(y_{i}\right)_{i\in\Z}$. The use of subscripts with `:' refers to certain vector components $\bm y_{1:k}=\left(y_{i}\right)_{1\leq i\leq k}$, or elements of a sequence $X_{1:k}=\left(X_{t}\right)_{1\leq t\leq k}$. For a random variable $U$ distributed as $\nu$ and a measurable function $h$, $\nu[h(U)]$ or simply $\nu[h]$ stands for the expectation of $h(U)$: $\nu[h]=\int h(u)\nu(\rmd u)$.

\section{Problem statement and main assumptions} \label{section:statement_of_the_problem_and_notation}
Real stable autoregressive processes of a fixed order, referred to as AR$(d)$ processes, are one of the simplest examples of CBS. They are defined as the stationary solution of 
\begin{eqnarray}
X_t &=& \sum\limits_{j=1}^{d}\theta_{j}X_{t-j}+\sigma\xi_{t} \label{equation:AR_model} \;,
\end{eqnarray}
where the $(\xi_{t})_{t\in\Z}$ are i.i.d. real random variables with $\EE[\xi_{t}]=0$ and $\EE[\xi_{t}^{2}]=1$.

We dispose of several efficient estimates for the parameter $\btheta=\begin{bmatrix}
\theta_{1} &\dots& \theta_{d} \end{bmatrix}^\prime$ which can be calculated via simple algorithms as Levinson-Durbin or Burg algorithm for example. From them we derive also efficient predictors. However, as the model is simple to handle, we use it to progressively introduce our general setup. 

Denote 
$$
A\left(\btheta\right) =
\begin{bmatrix}
  \theta_1 &\theta_2 &\dots&\dots &\theta_d \\
  1 &0&\dots&\dots&0 \\
  0 &1&0 &\ddots&0 \\
  \vdots&0&\ddots&\ddots&\vdots\\
  0&\dots&0&1&0
\end{bmatrix}\;,
$$
$\boldsymbol{X}_{t-1}=\begin{bmatrix}
X_{t-1} &\dots& X_{t-d} \end{bmatrix}^\prime$ and $\bm e_{1}=\begin{bmatrix}
1 & 0 &\dots& 0\end{bmatrix}^\prime$ the first canonical vector of $\R^{d}$. $M'$ represents the transpose of matrix $M$ (including vectors). The recurrence (\ref{equation:AR_model}) gives
\begin{eqnarray}
X_{t} = \btheta'\boldsymbol{X}_{t-1} + \sigma\xi_{t} = \sigma\sum\limits_{j=0}^{\infty}\bm e'_{1}A^{j}\left(\btheta\right)\bm e_{1}\xi_{t-j}\;. \label{equation:AR_X_xi}
\end{eqnarray}
The eigenvalues of $A\left(\btheta\right)$ are the inverses of the roots of the autoregressive polynomial $\btheta\left(z\right) = 1-\sum_{k=1}^{d}\theta_{k}z^{k}$, then at most $\delta$ for some $\delta\in\left(0,1\right)$ due to the stability of $X$ (see \cite{Brockwell_Davis:2006}). In other words $\btheta\in s_{d}\left(\delta\right)=\{\btheta~:~\btheta\left(z\right)\neq0\;\mbox{for}\;|z|<\delta^{-1}\}\subseteq s_{d}\left(1\right)$. In this context (or even in a more general one, see \cite{Kunsch:1995}) for all $\delta_{1}\in(\delta,1)$ there is a constant $\bar{K}$ depending only on $\btheta$ and $\delta_{1}$ such that for all $j\geq 0$
\begin{eqnarray}
\left|\bm e'_{1}A^{j}\left(\btheta\right)\bm e_{1}\right| \leq \bar{K}\delta_{1}^{j}\;, \label{equation:bound_AR_coefficient}
\end{eqnarray}
and then, the variance of $X_{t}$, denoted $\gamma_{0}$, satisfies 
$\gamma_{0} = \sigma^{2}\sum_{j=0}^{\infty}|\bm e'_{1}A^{j}\left(\btheta\right)\bm e_{1}|^{2} \leq \bar{K}^{2}\sigma^{2}/(1-\delta_{1}^{2})$.

The following definition allows to introduce the process which interests us. 
\begin{definition}
Let $\mathcal{X}'\subseteq\R^{r'}$ for some $r'\geq1$ and let $A=(A_{j})_{j\geq 0}$ be a sequence of non-negative numbers. A function $H: (\mathcal{X}')^{\N}\rightarrow\mathcal{X}$ is said to be $A$-Lipschitz if 
\begin{eqnarray*}
\|H\left(\bm u\right)-H\left(\bm v\right)\| &\leq& \sum\limits_{j=0}^{\infty}A_{j}\|u_{j}-v_{j}\|\;,
\end{eqnarray*}
for any $\bm u=(u_{j})_{j\in\N},
\bm v=(v_{j})_{j\in\N}\in(\mathcal{X}')^{\N}$.
\end{definition}

Provided $A=(A_{j})_{j\geq 0}$ with $A_{j}\geq0$ for all $j\geq0$, the i.i.d. sequence of
$\mathcal{X}'$-valued random variables $\left(\xi_{t}\right)_{t\in\Z}$ and $H:
(\mathcal{X}')^{\N}\rightarrow\mathcal{X}$, we consider that a time series $X~=~\left(X_{t}\right)_{t\in\Z}$ admitting the following property is a Causal Bernoulli Shift (CBS) with Lipschitz coefficients $A$ and innovations $\left(\xi_{t}\right)_{t\in\Z}$.
  
\begin{hypothesis*}{M}  
\item \label{hyp:CBS} The process $X=\left(X_{t}\right)_{t\in\Z}$ meets the
  representation
\begin{eqnarray*}
X_{t} &=& H\left(\xi_{t},\xi_{t-1},\xi_{t-2},\ldots\right), \forall t\in\Z\;,
\end{eqnarray*}
where $H$ is an $A$-Lipschitz function with the sequence $A$ satisfying
\begin{eqnarray}
\tilde{A}_{*} &=& \sum_{j=0}^{\infty}jA_{j} < \infty \;. \label{equation:tilde_A} 
\end{eqnarray}
We additionally define 
\begin{equation}
  A_{*} =\sum_{j=0}^{\infty}A_{j}\;. \label{equation:A}
\end{equation}
\end{hypothesis*}

CBS regroup several types of nonmixing stationary Markov chains, real-valued functional autoregressive models and Volterra processes, among other interesting models (see \cite{Coulon-Prieur_Doukhan:2000}). Thanks to the representation (\ref{equation:AR_X_xi}) and the inequality (\ref{equation:bound_AR_coefficient})  we assert that AR$(d)$ processes are CBS with $A_{j}=\sigma\bar{K}\delta_{1}^{j}$ for $j\geq 0$.

We let $\xi$ denote a random variable distributed as the $\xi_{t}$s. Results from \cite{Alquier_Li:2012} and \cite{Alquier_Wintenberger:2012} need a control on the exponential moment of $\xi$ in $\zeta=A_{*}$, which is provided via the following hypothesis. 
\begin{hypothesis*}{I}  
\item \label{hyp:exponential-moment} The innovations $\left(\xi_{t}\right)_{t\in\Z}$ satisfy $\phi(\zeta)=\EE\left[\rme^{\zeta\|\xi\|}\right]<\infty$.
\end{hypothesis*}
Bounded or Gaussian innovations trivially satisfy this hypothesis for any $\zeta\in\R$.

Let $\pi_{0}$ denote the probability distribution of the time series $Y$ that we aim to forecast. Observe that for a CBS, $\pi_{0}$ depends only on $H$ and the distribution of $\xi$. For any $f:\mathcal{X}^{\N^{*}}\rightarrow\mathcal{X}$ measurable and $t\in\Z$ we consider $\widehat{Y}_{t}=f\left(\left(Y_{t-i}\right)_{i\geq 1}\right)$, a possible predictor of $Y_{t}$ from its past. For a given loss function $\ell:\mathcal{X}\times\mathcal{X}\to\R_{+}$, the prediction risk is evaluated by the expectation of $\ell(\widehat{Y}_{t},Y_{t})$
\begin{eqnarray*}
R\left(f\right) = \EE\left[\ell\left(\widehat{Y}_{t},Y_{t}\right)\right] =
\pi_{0}\left[\ell\left(\widehat{Y}_{t},Y_{t}\right)\right] = \int\limits_{\mathcal{X}^{\Z}}\ell\left(f\left(\left(y_{t-i}\right)_{i\geq 1}\right),y_{t}\right)\pi_{0}\left(\rmd\bm y\right)\;.
\end{eqnarray*}
We assume in the following that the loss function $\ell$ fulfills the condition:
\begin{hypothesis*}{L}  
\item \label{assumption:Lipschitz_loss} For all $\bm y, \bm z\in\mathcal{X}$, $\ell\left(\bm y, \bm z\right) = g\left(\bm y-\bm z\right)$, for some convex function $g$ which is non-negative, $g\left(0\right)=0$ and $K$- Lipschitz: $\left|g\left(\bm y\right)-g\left(\bm z\right)\right| \leq K\|\bm y-\bm z\|$.
\end{hypothesis*}

If $\mathcal{X}$ is a subset of $\R$, $\ell\left(y,z\right)=\left|y-z\right|$ satisfies \ref{assumption:Lipschitz_loss} with $K=1$. 

From estimators of dimension $d$ for $\btheta$ we can build the corresponding linear predictors $f_{\btheta}\left(\bm y\right)=\btheta'\bm y_{1:d}$. Speaking more broadly, consider a set $\Theta$ and associated with it a set of predictors $\left\{f_{\btheta},\btheta\in\Theta\right\}$. For each $\btheta\in\Theta$ there is a unique $d=d\left(\btheta\right)\in\N^{*}$ such that $f_{\btheta}
:\mathcal{X}^{d}\rightarrow\mathcal{X}$ is a measurable function from
which we define 
\begin{eqnarray*}
\widehat{Y}_{t}^{\btheta} &=& f_{\btheta}\left(Y_{t-1},\ldots,Y_{t-d}\right) \;,
\end{eqnarray*}
as a predictor of $Y_{t}$ given its past. We can extend all functions $f_{\btheta}$ in a trivial way (using dummy variables) to start from $\mathcal{X}^{\N^{*}}$. A natural way to evaluate the predictor associated with $\btheta$ is to compute the risk $R\left(\btheta\right) = R\left(f_{\btheta}\right)$. We use the same letter $R$ by an abuse of notation. 

We observe $X_{1:T}$ from $X=\left(X_{t}\right)_{t\in\Z}$, an independent copy of $Y$. A crucial goal of this work is to build a predictor function $\hat{f}_{T}$ for $Y$, inferred from the sample $X_{1:T}$ and $\Theta$ such that $R(\hat{f}_{T})$ is close to $\inf_{\btheta\in\Theta}R\left(\btheta\right)$ with $\pi_{0}$- probability close to $1$. 

The set $\Theta$ also depends on $T$, we write $\Theta\equiv\Theta_{T}$. Let us define
\begin{eqnarray}
d_{T}=\sup_{\btheta\in\Theta_{T}}d\left(\btheta\right)\;. \label{equation:d_T} 
\end{eqnarray}

The main assumptions on the set of predictors are the following ones.

\begin{hypothesis}{P}  
\item \label{assumption:Lipschitz}
The set $\left\{f_{\btheta},\btheta\in\Theta_{T}\right\}$ is such that for any $\btheta\in\Theta_{T}$ there are $b_{1}\left(\btheta\right),\ldots,$ $b_{d\left(\btheta\right)}\left(\btheta\right)~\in~\R_{+}$ satisfying for all $\bm y=\left(y_{i}\right)_{i\in\N^{*}},\bm z=\left(z_{i}\right)_{i\in\N^{*}}\in\mathcal{X}^{\N^{*}}$,
\begin{eqnarray*}
\left|\left|f_{\btheta}(\bm y)-f_{\btheta}(\bm z)\right|\right| &\leq& \sum\limits_{j=1}^{d\left(\btheta\right)}b_{j}(\btheta)\left|\left|y_{j}-z_{j}\right|\right| \;.
\end{eqnarray*}
We assume moreover that $L_{T}=\sup_{\btheta\in\Theta_{T}}\sum_{j=1}^{d\left(\btheta\right)}b_{j}\left(\btheta\right)<\infty$. 
\item \label{assumption:L_log} The inequality $L_{T}+1\leq \log T$ holds for all $T\geq 4$.
\end{hypothesis}

In the case where $\mathcal{X}\subseteq\R$ and $\left\{f_{\btheta},\btheta\in\Theta_{T}\right\}$ is such that $\btheta\in\R^{d\left(\btheta\right)}$ and $f_{\btheta}\left(\bm y\right)=\btheta'\bm y_{1:d(\btheta)}$ for all $\bm y\in\R^{\N}$, we have 
\begin{eqnarray*}
\left|f_{\btheta}(\bm y)-f_{\btheta}(\bm z)\right| &\leq& \sum\limits_{j=1}^{d\left(\btheta\right)}\left|\theta_{j}\right|\left|y_{j}-z_{j}\right| \;.
\end{eqnarray*}
The last conditions are satisfied by the linear predictors when $\Theta_{T}$ is a subset of the $\ell_{1}$-ball of radius $\log T-1$ in $\R^{d_{T}}$ .

\section{Prediction via aggregation} \label{section:prediction_via_aggregation}
The predictor that we propose is defined as an average of predictors $f_{\btheta}$ based on the empirical version of the risk,
\begin{eqnarray*}
r_{T}\left(\btheta\left|X\right.\right) &=& \frac{1}{T-d\left(\btheta\right)}\sum\limits_{t=d\left(\btheta\right)+1}^{T}\ell\left(\widehat{X}_{t}^{\btheta},X_{t}\right) \;.
\end{eqnarray*}
where $\widehat{X}_{t}^{\btheta} = f_{\btheta}\left(\left(X_{t-i}\right)_{i\geq 1}\right)$. The function $r_{T}\left(\btheta\left|X\right.\right)$ relies on $X_{1:T}$ and can be computed at stage $T$; this is in fact a statistic. 

We consider a prior probability measure $\pi_{T}$ on $\Theta_{T}$. The prior serves to control the complexity of predictors associated with $\Theta_{T}$. Using $\pi_{T}$ we can construct one predictor in particular, as detailed in the following.

\subsection{Gibbs predictor} \label{subsection:Gibbs_predictor}
For a measure $\nu$ and a measurable function $h$ (called energy function) such that 
$\nu\left[\exp\left(h\right)\right]= \int \exp\left(h\right) \;\rmd\nu<\infty\;,$
we denote by $\nu\left\{h\right\}$ the measure defined as
\begin{eqnarray*}
\nu\left\{h\right\}\left(\rmd\btheta\right) &=&
\frac{\exp\left(h\left(\btheta\right)\right)}{\nu\left[\exp\left(h\right)\right]}\nu\left(\rmd\btheta\right) \;. 
\end{eqnarray*}
It is known as the Gibbs measure.

\begin{definition}[Gibbs predictor]
Given $\eta>0$, called the temperature or the learning rate parameter, we define the Gibbs predictor as the expectation of $f_{\btheta}$, where $\btheta$ is drawn under $\pi_{T}\left\{-\eta r_{T}\left(\cdot\left|X\right.\right)\right\}$, that is
\begin{eqnarray}
\hat{f}_{\eta,T}\left(\bm y\left|X\right.\right) = \pi_{T}\left\{-\eta r_{T}\left(\cdot\left|X\right.\right)\right\}\left[f_{\cdot}\left(\bm y\right)\right] = \int\limits_{\Theta_{T}}f_{\btheta}\left(\bm y\right)\frac{\exp\left(-\eta r_{T}\left(\btheta\left|X\right.\right)\right)}{\pi_{T}\left[\exp\left(-\eta r_{T}\left(\cdot\left|X\right.\right)\right)\right]}\pi_{T}\left(\rmd\btheta\right)\;. \label{equation:Gibbs_predictor}
\end{eqnarray}
\end{definition}

\subsection{PAC-Bayesian inequality} \label{subsection:PAC-Bayesian_inequality}

At this point more care must be taken to describe  $\Theta_{T}$. Here and in the following we suppose that 
\begin{eqnarray}
\Theta_{T}\subseteq\R^{n_{T}}\,\,\,\textrm{for some } n_{T}\in\N^{*}\;. \label{equation:n_T}
\end{eqnarray}
Suppose moreover that $\Theta_{T}$ is equipped with the Borel $\sigma$-algebra $\mathcal{B}(\Theta_{T})$.

A Lipschitz type hypothesis on $\btheta$ guarantees the robustness of the set $\left\{f_{\btheta}, \btheta\in\Theta_{T}\right\}$ with respect to the risk $R$.

\begin{hypothesis}{P}  
\item \label{assumption:theta_Lipschitz} 
There is $\mathcal{D}<\infty$ such that for all $\btheta_{1},\btheta_{2}\in\Theta_{T}$,
\begin{eqnarray*}
\pi_{0}\left[\left|\left|f_{\btheta_{1}}\left(\left(X_{t-i}\right)_{i\geq 1}\right)-f_{\btheta_{2}}\left(\left(X_{t-i}\right)_{i\geq 1}\right)\right|\right|\right] &\leq& \mathcal{D}d_{T}^{1/2}\left|\left|\btheta_{1}-\btheta_{2}\right|\right|\;,
\end{eqnarray*}
where $d_{T}$ is defined in (\ref{equation:d_T}). 
\end{hypothesis}
Linear predictors satisfy this last condition with $\mathcal{D}=\pi_{0}\left[\left|X_{1}\right|\right]$.

Suppose that the $\btheta$ reaching the $\inf_{\btheta\in\Theta_{T}}R(\btheta)$ has some zero components, i.e. $\mathrm{supp}(\btheta)<n_{T}$. Any prior with a lower bounded density (with respect to the Lebesgue measure) allocates zero mass on lower dimensional subsets of $\Theta_{T}$. Furthermore, if the density is upper bounded we have $\pi_{T}[B(\btheta,\Delta)\cap\Theta_{T}]=O(\Delta^{n_{T}})$ for $\Delta$ small enough. As we will notice in the proof of Theorem~\ref{theorem:oracle_Gibbs}, a bound like the previous one would impose a tighter constraint to $n_{T}$. Instead we set the following condition.
\begin{hypothesis}{P}  
\item \label{assumption:theta_balls} 
There is a sequence $\left(\btheta_{T}\right)_{T\geq 4}$ and constants $\mathcal{C}_{1}>0$, $\mathcal{C}_{2},\mathcal{C}_{3}\in(0,1]$ and $\gamma\geq 1$ such that $\btheta_{T}\in\Theta_{T}$,
\begin{eqnarray*}
R\left(\btheta_{T}\right) &\leq& \inf\limits_{\btheta\in\Theta_{T}} R\left(\btheta\right)+\mathcal{C}_{1}\frac{\log^{3}\left(T\right)}{T^{1/2}}\;, \\
\textrm{and}\;\;\;\pi_{T}\left[B\left(\btheta_{T},\Delta\right)\cap\Theta_{T}\right] &\geq& \mathcal{C}_{2}\Delta^{n_{T}^{1/\gamma}}, \forall 0\leq\Delta\leq\Delta_{T}=\frac{\mathcal{C}_{3}}{T}\;.
\end{eqnarray*}
\end{hypothesis}
A concrete example is provided in Section~\ref{section:numerical_work}.

We can now present the main result of this section, our PAC-Bayesian inequality concerning the predictor $\hat{f}_{\eta_{T},T}\left(\cdot\left|X\right.\right)$ built following (\ref{equation:Gibbs_predictor}) with the learning rate $\eta~=~\eta_{T}~=~T^{1/2}/(4\log T)$, provided an arbitrary probability measure $\pi_{T}$ on $\Theta_{T}$.

\begin{theorem}\label{theorem:oracle_Gibbs}
  Let $\ell$ be a loss function such that Assumption
  \ref{assumption:Lipschitz_loss} holds. Consider a process
  $X=\left(X_{t}\right)_{t\in\Z}$ satisfying
  Assumption~\ref{hyp:CBS} and let $\pi_0$ denote its probability 
  distribution. Assume that the innovations fulfill Assumption~\ref{hyp:exponential-moment} with $\zeta=A_{*}$; $A_{*}$ is defined in 
(\ref{equation:A}). For each $T\geq 4$ let
  $\left\{f_{\btheta},\btheta\in\Theta_{T}\right\}$ be a set of predictors
  meeting Assumptions \ref{assumption:Lipschitz}, \ref{assumption:L_log} and
  \ref{assumption:theta_Lipschitz} such that $d_{T}$, defined in
  (\ref{equation:d_T}), is at most $T/2$. Suppose that the set $\Theta_{T}$ is
  as in (\ref{equation:n_T}) with $n_{T}\leq\log^{\gamma}T$ for some
  $\gamma\geq 1$ and we let $\pi_{T}$ be a probability measure on it such that
  Assumption~\ref{assumption:theta_balls} holds for the same $\gamma$. Then for
  any $\varepsilon>0$, with $\pi_{0}$-probability at least $1-\varepsilon$,
\begin{eqnarray*}
R\left(\hat{f}_{\eta_{T},T}\left(\cdot\left|X\right.\right)\right) &\leq&
\inf\limits_{\btheta\in\Theta_{T}}R\left(f_{\btheta}\right)+\mathcal{E}\frac{\log^{3}T}{T^{1/2}}+\frac{8\log T}{T^{1/2}}\log\left(\frac{1}{\varepsilon}\right)\;,
\end{eqnarray*}
where
\begin{align}
\mathcal{E} = \mathcal{C}_{1}+8+\frac{2}{\log 2}-\frac{2\log\mathcal{C}_{2}}{\log^{2}2}-\frac{4\log\mathcal{C}_{3}}{\log 2}+ \frac{8K^{2}\left(A_{*}+\tilde{A}_{*}\right)^{2}}{\tilde{A}_{*}^{2}}+\frac{K\mathcal{D}\mathcal{C}_{3}}{8\log^{3}2} \nonumber \\
+\frac{4K\phi(A_{*})}{\log 2}+\frac{2K^{2}\phi(A_{*})}{\log^{2}2}\;, \label{equation:constant_theorem_oracle_Gibbs}
\end{align}
with $\tilde{A}_{*}$ defined in (\ref{equation:tilde_A}),  $K$, $\phi$ and $\mathcal{D}$ in
Assumptions \ref{assumption:Lipschitz_loss},
\ref{hyp:exponential-moment} and \ref{assumption:theta_Lipschitz},
respectively, and $\mathcal{C}_{1}$, $\mathcal{C}_{2}$ and $\mathcal{C}_{3}$
in Assumption~\ref{assumption:theta_balls}.
\end{theorem}
The proof is postponed to Section~\ref{subsection:proof_theorem_oracle_Gibbs}. 

Here however we insist on the fact that this inequality applies to an exact aggregated predictor $\hat{f}_{\eta_{T},T}\left(\cdot\left|X\right.\right)$. We need to investigate how these predictors are computed and how practical numerical approximations behave compared to the properties of the exact version. 

\section{Stochastic approximation} \label{section:stochastic_approximation}
Once we have the observations $X_{1:T}$, we use the Metropolis - Hastings algorithm to compute $\hat{f}_{\eta,T}\left(\cdot\left|X\right.\right)=\int f_{\btheta}\left(\cdot\left|X\right.\right)\pi_{T}\left\{-\eta r_{T}\left(\btheta\left|X\right.\right)\right\}\left(\rmd\btheta\right)$. The Gibbs measure $\pi_{T}\left\{-\eta r_{T}\left(\cdot\left|X\right.\right)\right\}$ is a distribution on $\Theta_{T}$ whose density $\pi_{\eta,T}\left(\cdot\left|X\right.\right)$ with respect to $\pi_{T}$ is proportional to $\exp\left(-\eta r_{T}\left(\cdot\left|X\right.\right)\right)$.

\subsection{Metropolis - Hastings algorithm} \label{subsection:Metropolis-Hastings_algorithm}
Given $X\in\mathcal{X}^{\Z}$, the
Metropolis-Hastings algorithm generates a Markov chain $\Phi_{\eta,T}\left(X\right)=(\btheta_{\eta,T,n}(X))_{n\geq 0}$ with kernel $P_{\eta,T}$ (only depending on $X_{1:T}$) having the target distribution $\pi_{T}\left\{-\eta r_{T}\left(\cdot\left|X\right.\right)\right\}$ as
the unique invariant measure, based on the transitions of another Markov chain which serves as a
proposal (see \cite{Roberts_Rosenthal:2004}). We consider a proposal transition of the form $Q_{\eta,T}(\btheta_{1},\rmd\btheta)=q_{\eta,T}(\btheta_{1},\btheta)\pi_{T}(\rmd\btheta)$ where the conditional density kernel $q_{\eta,T}$ (possibly also depending on $X_{1:T}$) on $\Theta_{T}\times\Theta_{T}$ is such that 
\begin{eqnarray}
\beta_{\eta,T}\left(X\right) = \inf\limits_{\left(\btheta_{1},\btheta_{2}\right)\in\Theta_{T}\times\Theta_{T}}\frac{q_{\eta,T}\left(\btheta_{1},\btheta_{2}\right)}{\pi_{\eta,T}\left(\btheta_{2}\left|X\right.\right)}\in\left(0,1\right)\;. \label{equation:convergence}
\end{eqnarray}
This is the case of the independent Hastings algorithm, where the proposal is i.i.d. with
  density $q_{\eta,T}$. The condition gets into
\begin{eqnarray}
\beta_{\eta,T}\left(X\right) = \inf\limits_{\btheta\in\Theta_{T}}\frac{q_{\eta,T}\left(\btheta\right)}{\pi_{\eta,T}\left(\btheta\left|X\right.\right)}\in\left(0,1\right) \;. \label{equation:convergence_independent_Hastings}
\end{eqnarray}
In Section~\ref{section:numerical_work} we provide an example.

The relation (\ref{equation:convergence}) implies that the algorithm is uniformly ergodic, i.e. we have a control in total variation norm ($\|\cdot\|_{TV}$). Thus, the following condition holds (see \cite{Mengersen_Tweedie:1996}).

\begin{hypothesis*}{A}  
\item \label{assumption:geometric_convergence}
Given $\eta,T>0$, there is $\beta_{\eta,T}:\mathcal{X}^{\Z}\rightarrow\left(0,1\right)$ such for any $\btheta_{0}\in\Theta_{T}$, $\bm x\in\mathcal{X}^{\Z}$ and $n\in\N$, the chain $\Phi_{\eta,T}\left(\bm x\right)$ with transition law $P_{\eta,T}$ and invariant distribution $\pi_{T}\left\{-\eta r_{T}\left(\cdot\left|\bm x\right.\right)\right\}$ satisfies
\begin{eqnarray*}
\left|\left|P_{\eta,T}^{n}\left(\btheta_{0},\cdot\right)-\pi_{T}\left\{-\eta r_{T}\left(\cdot\left|\bm x\right.\right)\right\}\right|\right|_{TV} &\leq& 2\left(1-\beta_{\eta,T}\left(\bm x\right)\right)^{n}\;.
\end{eqnarray*}
\end{hypothesis*}

\subsection{Theoretical bounds for the computation}
In \cite[Theorem 3.1]{Latuszynski_Niemiro:2011} we find a bound on the mean square error of approximating one integral by the empirical estimate obtained from the successive samples of certain ergodic Markov chains, including those generated by the MCMC method that we use.

A MCMC method adds a second source of randomness to the forecasting process and our aim is to measure it. Let $\btheta_{0}\in\cap_{T\geq 1}\Theta_{T}$, we set $\btheta_{\eta,T,0}\left(\bm x\right)=\btheta_{0}$ for all $T,\eta>0$, $\bm x\in\mathcal{X}^{\Z}$. We denote by $\mu_{\eta,T}\left(\cdot\left|X\right.\right)$ the probability distribution of the Markov chain $\Phi_{\eta,T}\left(X\right)$ with initial point $\btheta_{0}$ and kernel $P_{\eta,T}$.

Let $\nu_{\eta,T}$ denote the probability distribution of $(X,\Phi_{\eta,T}\left(X\right))$; it is defined by setting for all sets $A\in(\mathcal{B}(\mathcal{X}))^{\otimes\Z}$ and $B\in(\mathcal{B}(\Theta_{T}))^{\otimes\N}$
\begin{eqnarray}
\nu_{\eta,T}\left(A\times B\right) = \int\1_{A}\left(\bm x\right)\1_{B}\left(\bm\phi\right)\mu_{\eta,T}\left(\rmd\bm\phi\left|\bm x\right.\right)\pi_{0}\left(\rmd\bm x\right) \label{equation:X_Phi_distribution}
\end{eqnarray}
Given $\Phi_{\eta,T}=(\btheta_{\eta,T,n})_{n\geq 0}$, we then define for $n\in\N^{*}$
\begin{eqnarray}
\bar{f}_{\eta,T,n}=\frac{1}{n}\sum_{i=0}^{n-1}f_{\btheta_{\eta,T,i}}\,.\label{equation:numerical_Gibbs_predictor}
\end{eqnarray}
Since our chain depends on $X$, we make it explicit by using the notation $\bar{f}_{\eta,T,n}\left(\cdot\left|X\right.\right)$. The cited \cite[Theorem 3.1]{Latuszynski_Niemiro:2011} leads to a proposition that applies to the numerical approximation of the Gibbs predictor (the proof is in Section~\ref{subsection:proof_proposition_numerical_oracle_Gibbs}). We stress that this is independent of the model (CBS or any), of the set of predictors and of the theoretical guarantees of Theorem~\ref{theorem:oracle_Gibbs}.
\begin{proposition} \label{proposition:numerical_oracle_Gibbs}
Let $\ell$ be a loss function meeting Assumption~\ref{assumption:Lipschitz_loss}.
Consider any process $X=\left(X_{t}\right)_{t\in\Z}$ with an arbitrary probability distribution $\pi_{0}$. Given $T\geq 2$, $\eta>0$, a set of predictors $\left\{f_{\btheta},\btheta\in\Theta_{T}\right\}$ and $\pi_{T}\in\mathcal{M}_{+}^{1}\left(\Theta_{T}\right)$, let $\hat{f}_{\eta,T}\left(\cdot\left|X\right.\right)$ be defined by (\ref{equation:Gibbs_predictor}) and let $\bar{f}_{\eta,T,n}\left(\cdot\left|X\right.\right)$ be defined by (\ref{equation:numerical_Gibbs_predictor}). Suppose that $\Phi_{\eta,T}$ meets Assumption~\ref{assumption:geometric_convergence} for $\eta$ and $T$ with a function $\beta_{\eta,T}:\mathcal{X}^{\Z}\rightarrow(0,1)$. Let $\nu_{\eta,T}$ denote the probability distribution of $(X,\Phi_{\eta,T}\left(X\right))$ as defined in (\ref{equation:A_proposition_numerical_oracle_Gibbs}). Then, for all $n\geq 1$ and $D>0$, with $\nu_{\eta,T}$- probability at least $\max\{0,1-A_{\eta,T}/(Dn^{1/2})\}$ we have $|R(\bar{f}_{\eta,T,n}\left(\cdot\left|X\right.\right))-R(\hat{f}_{\eta,T}\left(\cdot\left|X\right.\right))|\leq D$, where
\begin{eqnarray}
A_{\eta,T} = 3K\int\limits_{\mathcal{X}^{\Z}}\frac{1}{\beta_{\eta,T}\left(\bm x\right)}\int\limits_{\mathcal{X}^{\Z}}\sup\limits_{\btheta\in\Theta_{T}}\left|f_{\btheta}\left(\bm y\right)-\hat{f}_{\eta,T}\left(\bm y\left|\bm x\right.\right)\right|\pi_{0}\left(\rmd\bm y\right)\pi_{0}\left(\rmd\bm x\right)\,. \label{equation:A_proposition_numerical_oracle_Gibbs}
\end{eqnarray}
\end{proposition}
We denote by $\nu_{T}=\nu_{\eta_{T},T}$ the probability distribution of $(X,\Phi_{\eta,T}\left(X\right))$ setting $\eta~=~\eta_{T}~=~T^{1/2}/(4\log T)$. As Theorem~\ref{theorem:oracle_Gibbs} does not involve any simulation, it also holds in $\nu_{T}$- probability. From this and Proposition~\ref{proposition:numerical_oracle_Gibbs} a union bound gives us the following.

\begin{theorem} \label{theorem:oracle_numerical_Gibbs}
Under the hypothesis of Theorem~\ref{theorem:oracle_Gibbs}, consider moreover that Assumption~\ref{assumption:geometric_convergence} is fulfilled by $\Phi_{\eta,T}$ for all $\eta=\eta_{T}$ and $T$ with $T\geq 4$. Thus, for all $\varepsilon>0$ and $n\geq
M\left(T,\varepsilon\right)$, with $\nu_{T}$- probability at least $1-\varepsilon$ we have
\begin{eqnarray*}
R\left(\bar{f}_{\eta_{T},T,n}\left(\cdot\left|X\right.\right)\right)  &\leq& \inf\limits_{\btheta\in\Theta_{T}}R\left(f_{\btheta}\right)+\left(\mathcal{E}+\frac{2}{\log 2}+2\right)\frac{\log^{3}T}{T^{1/2}}+\frac{8\log T}{T^{1/2}}\log\left(\frac{1}{\varepsilon}\right) \;,
\end{eqnarray*}
where $\mathcal{E}$ is defined in (\ref{equation:constant_theorem_oracle_Gibbs}) and $M\left(T,\varepsilon\right)=A_{\eta_{T},T}^{2}T/(\varepsilon^{2}\log^{6}T)$ with $A_{\eta,T}$ as in (\ref{equation:A_proposition_numerical_oracle_Gibbs}).
\end{theorem}

\section{Applications to the autoregressive process} \label{section:numerical_work}
We carefully recapitulate all the assumptions of Theorem~\ref{theorem:oracle_numerical_Gibbs} in the context of an autoregressive process. After that, we illustrate numerically the behaviour of the proposed method.

\subsection{Theoretical considerations}
Consider a real valued stable autoregressive process of finite order $d$ as defined by (\ref{equation:AR_model}) with parameter $\btheta$ lying in the interior of $s_{d}\left(\delta\right)$ and unit normally distributed innovations (Assumptions \ref{hyp:CBS} and \ref{hyp:exponential-moment} hold). With the loss function $\ell\left(y,z\right)=\left|y-z\right|$ Assumption~\ref{assumption:Lipschitz_loss} holds as well. The linear predictors is the set that we test; they meet Assumption~\ref{assumption:theta_Lipschitz}. Without loss of generality assume that $d_{T}=n_{T}$. In the described framework we have $
\hat{f}_{\eta,T}\left(\cdot\left|X\right.\right) = f_{\hat{\btheta}_{\eta,T}\left(X\right)}$, where
$$
\hat{\btheta}_{\eta,T}\left(X\right) = \int\limits_{\Theta_{T}}\btheta\frac{\exp\left(-\eta r_{T}\left(\btheta\left|X\right.\right)\right)}{\pi_{T}\left[\exp\left(-\eta r_{T}\left(\btheta\left|X\right.\right)\right)\right]}\pi_{T}\left(\rmd\btheta\right)\,.
$$
This $\hat{\btheta}_{\eta,T}\left(X\right)\in\R^{d_{T}}$ is known as the Gibbs estimator.

Remark that, by (\ref{equation:AR_X_xi}) and the normality of the innovations, the risk of any $\hat{\btheta}\in\R^{d_{T}}$ is computed as the absolute moment of a centered Gaussian, namely 
\begin{eqnarray}
R\left(f_{\hat{\btheta}}\right)=R\left(\hat{\btheta}\right)=\frac{\left(2\left(\hat{\btheta}-\btheta\right)'\Gamma_{T}\left(\hat{\btheta}-\btheta\right)+2\sigma^{2}\right)^{1/2}}{\pi^{1/2}}\;, \label{equation:risk_AR_l_1}
\end{eqnarray}
where $\Gamma_{T}=(\gamma_{i,j})_{0\leq i,j\leq d_{T}-1}$ is the covariance matrix of the process. In (\ref{equation:risk_AR_l_1}) the vector $\btheta$ originally in $\R^{d}$ is completed by $d_{T}-d$ zeros.

In this context $\arg\inf_{\btheta\in\R^{\N^{*}}}R\left(\btheta\right)\in s_{d}(1)$ gives the true parameter $\btheta$ generating the process. Let us verify Assumption~\ref{assumption:theta_balls} by setting conveniently $\Theta_{T}$ and $\pi_{T}$.
Let $\Delta_{d*}>0$ be such that $B\left(\btheta,\Delta_{d*}\right)\subseteq s_{d}(1)$.

We express $\Theta_{T}=\bigcup_{k=1}^{d_{T}}\Theta_{k,T}$ where $\btheta\in\Theta_{k,T}$ if and only if $d\left(\btheta\right)=k$. It is interesting to set $\Theta_{k,T}$ as the part of the stability domain of an AR$(k)$ process satisfying Assumptions \ref{assumption:Lipschitz} and \ref{assumption:L_log}. Consider $\Theta_{1,T}=s_{1}(1)\times\{0\}^{d_{T}-1}\cap B_{1}\left(\bm 0,\log T-1\right)$ and $\Theta_{k,T}=s_{k}(1)\times\{0\}^{d_{T}-k} \cap B_{1}\left(\bm 0,\log T-1\right)\backslash\Theta_{k-1,T}$ for $k\geq 2$. Assume moreover that $d_{T}=\lfloor\log^{\gamma}T\rfloor$. 

We write $\pi_{T}=\sum_{k=1}^{d_{T}}c_{k,T}\pi_{k,T}$ where for all 
$k$, $c_{k,T}\pi_{k,T}$ is the restriction of $\pi_{T}$ to $\Theta_{k,T}$ with $c_{k,T}$ a real non negative number and $\pi_{k,T}$ a probability measure on $\Theta_{k,T}$. In this setup $c_{k,T}=\pi_{T}\left[\Theta_{k,T}\right]$ and $\pi_{k,T}\left[A\cap\Theta_{k,T}\right]=\pi_{T}\left[A\cap\Theta_{k,T}\right]/c_{k,T}$ if $c_{k,T}>0$ and $\pi_{k,T}\left[A\cap\Theta_{k,T}\right]=0$ otherwise. The vector $\begin{bmatrix} c_{1,T} & \ldots &c_{d_{T},T} \end{bmatrix}$ could be interpreted as a prior on the model order. Set $c_{k,T}=c_{k}/(\sum_{i=1}^{d_{T}}c_{i})$ where $c_{k}>0$ is the $k$-th term of a convergent series ($\sum_{k=1}^{\infty}c_{k}=c^{*}<\infty$). 

The distribution $\pi_{k,T}$ is inferred from some transformations explained below. Observe first that if $a\leq b$ we have $s_{k}(a)\subseteq s_{k}(b)$. If $\btheta\in s_{k}(1)$ then $\begin{bmatrix} \lambda\theta_{1} & \ldots & \lambda^{k}\theta_{k}\end{bmatrix}^\prime\in s_{k}(1)$ for any $\lambda\in(-1,1)$. Let us set
$$
\lambda_{T}(\btheta) =  \min\left\{1,\frac{\log T-1}{\|\btheta\|_{1}}\right\}\;.
$$
We define $F_{k,T}(\btheta)=\begin{bmatrix} \lambda_{T}(\btheta)\theta_{1} & \ldots & \lambda_{T}^{k}(\btheta)\theta_{k} & 0 & \ldots & 0\end{bmatrix}^\prime\in\R^{d_{T}}$. Remark that for any $\btheta\in s_{k}(1)$, $\|F_{k,T}(\btheta)\|_{1}\leq \lambda_{T}(\btheta)\|\btheta\|_{1}\leq \log T-1$. This gives us an idea to generate vectors in $\Theta_{k,T}$. Our distribution $\pi_{k,T}$ is deduced from: \\
\\
\begin{algorithm}[H] \caption{$\pi_{k,T}$ generation} \label{algorithm:pi_kT_generation}
\KwIn{an effective dimension $k$, the number of observations $T$ and $F_{k,T}$}
generate a random $\btheta$ uniformly on $s_{k}(1)$\;
\Return{$F_{k,T}\left(\btheta\right)$}
\end{algorithm}
\vspace{0.3cm}
The distribution $\pi_{k,T}$ is lower bounded by the uniform distribution on $s_{k}(1)$.

Provided any $\gamma\geq 1$, let $T_{*}=\min\{T: d_{T}\geq d^{\gamma}, \log T\geq d^{1/2}2^{d}\}$. Since $s_{k}(1)\subseteq B(\bm 0, 2^{k}~-~1)$ (see \cite[Lemma 1]{Moulines_Priouret_Roueff:2005}) and $k^{1/2}\|\btheta\|\geq\|\btheta\|_{1}$ for any $\btheta\in\R^{k}$, the constraint $\|\btheta\|_{1}\leq\log T-1$  becomes redundant in $\Theta_{k,T}$ for $1\leq k\leq d$ and $T\geq T_{*}$, i.e. $\Theta_{1,T}=s_{1}(1)\times\{0\}^{d_{T}-1}$ and $\Theta_{k,T}=s_{k}(1)\times\{0\}^{d_{T}-k}\backslash\Theta_{k-1,T}$ for $2\leq k\leq d$. We define the sequence of Assumption~\ref{assumption:theta_balls} as $\btheta_{T}=\bm 0$ for $T<T_{*}$ and $\btheta_{T}=\arg\inf_{\btheta\in\Theta_{T}}R(\btheta)$ for $T\geq T_{*}$. Remark that the first $d$ components of $\btheta_{T}$ are constant for $T\geq T_{*}$ (they correspond to the $\btheta\in\R^{d}$ generating the AR$(d)$ process), and the last $d_{T}-d$ are zero. Let $\Delta_{1*}=2\log 2-1$. Then, we have for $T<T_{*}$ and all $\Delta\in[0,\Delta_{1*}]$
$$
\pi_{T}\left[B\left(\btheta_{T},\Delta\right)\cap\Theta_{T}\right]\geq c_{1,T}\pi_{1,T}\left[B\left(\bm 0,\Delta\right)\cap s_{1}(1)\times\{0\}^{d_{T}-1}\right] \geq \frac{c_{1}}{c^{*}}\Delta\;.
$$ 
Furthermore, for $T\geq T_{*}$ and $\Delta\in[0,\Delta_{d*}]$
$$
\pi_{T}\left[B\left(\btheta_{T},\Delta\right)\cap\Theta_{T}\right]\geq c_{d,T}\pi_{d,T}\left[B\left(\btheta_{T},\Delta\right)\cap s_{d}(1)\times\{0\}^{d_{T}-d}\right] \geq \frac{c_{d}}{2^{d^2}c^{*}}\Delta^{d}\;.
$$ 
Assumption~\ref{assumption:theta_balls} is then fulfilled for any $\gamma\geq 1$ with 
\begin{eqnarray*}
\mathcal{C}_{1}&=&\max\left\{0,(R\left(0\right)-\inf_{\btheta\in\Theta_{T}} R\left(\btheta\right))T^{1/2}\log^{-3}T, 4\leq T<T_{*}\right\} \\
\mathcal{C}_{2}&=&\min\left\{1,\frac{c_{1}}{c^{*}},\frac{c_{d}}{2^{d^{2}}c^{*}}\right\} \\
\mathcal{C}_{3}&=&\min\left\{1,4\Delta_{1*},T_{*}\Delta_{d*}\right\}\;.
\end{eqnarray*}
Let $q_{\eta,T}$ be the constant function $1$, this means that the proposal has the same distribution $\pi_{T}$. Let us bound the ratio (\ref{equation:convergence_independent_Hastings}).
\begin{eqnarray}
\beta_{\eta,T}\left(X\right) = \inf\limits_{\btheta\in\Theta_{T}}\frac{q_{\eta,T}\left(\btheta\right)}{\pi_{\eta,T}\left(\btheta\left|X\right.\right)} &=& \inf\limits_{\btheta\in\Theta_{T}}\frac{\displaystyle\sum\limits_{k=1}^{d_{T}}c_{k,T}\int\limits_{\Theta_{k,T}}\exp\left(-\eta r_{T}\left(z\left|X\right.\right)\right)\pi_{k,T}\left(\rmd\bm z\right)}{\exp\left(-\eta r_{T}\left(\btheta\left|X\right.\right)\right)} \nonumber\\
&\geq& \sum\limits_{k=1}^{d_{T}}c_{k,T}\int\limits_{\Theta_{k,T}}\exp\left(-\eta r_{T}\left(z\left|X\right.\right)\right)\pi_{k,T}\left(\rmd\bm z\right) > 0\;. \label{equation:bound_beta}
\end{eqnarray}
Now note that 
\begin{eqnarray}
\left|x_{t}-f_{\btheta}\left(\left(x_{t-i}\right)_{i\geq 1}\right)\right| \leq \left|x_{t}\right|+\sum\limits_{j=1}^{d\left(\btheta\right)}\left|\theta_{j}\right|\left|x_{t-j}\right| \leq \log T\max\limits_{j=0,\ldots,d\left(\btheta\right)}\left|x_{t-j}\right|\;. \label{equation:deterministic_bound_sum_max}
\end{eqnarray}
Plugging the bound (\ref{equation:deterministic_bound_sum_max}) on (\ref{equation:bound_beta}) with $\eta=\eta_{T}$
$$
\beta_{\eta_{T},T}\left(\bm x\right) \geq \sum\limits_{k=1}^{d_{T}}c_{k}\int\limits_{\Theta_{k}}\exp\left(-\eta_{T}r_{T}\left(z\left|\bm x\right.\right)\right)\pi_{k}\left(\rmd\bm z\right) \geq \exp\left(-\frac{T^{1/2}}{4}\max\limits_{j=0,\ldots,d_{T}}\left|x_{t-j}\right|\right)\;,
$$
we deduce that 
\begin{eqnarray}
\frac{1}{\beta_{\eta_{T},T}\left(\bm x\right)} \leq \sum\limits_{k=0}^{d_{T}} \exp\left(\frac{T^{1/2}\left|x_{t-j}\right|}{4}\right)\;. \label{equation:beta_bound}
\end{eqnarray}
Taking (\ref{equation:beta_bound}) into account, setting $\gamma=1$ (thus $d_{T}=\lfloor\log T\rfloor$), using Assumption~\ref{assumption:theta_Lipschitz}, that $K=1$ and applying the Cauchy-Schwarz inequality we get
\begin{align*}
A_{\eta_{T},T} &= 3K\int\limits_{\mathcal{X}^{\Z}}\frac{1}{\beta_{\eta_{T},T}\left(\bm x\right)}\int\limits_{\mathcal{X}^{\Z}}\sup\limits_{\btheta\in\Theta_{T}}\left|f_{\btheta}\left(\bm y\right)-f_{\hat{\btheta}_{\eta_{T},T}\left(\bm x\right)}\left(\bm y\right)\right|\pi_{0}\left(\rmd\bm y\right)\pi_{0}\left(\rmd\bm x\right)\\
&\leq 3\left(d_{T}+1\right)d_{T}^{1/2}\pi_{0}\left[\exp\left(\frac{T^{1/2}\left|X_{1}\right|}{4}\right)\right]\pi_{0}\left[\left|X_{1}\right|\right]\sup\limits_{\btheta\in\Theta_{T}}\left|\left|\btheta\right|\right| \\
&\leq 6\log^{3/2} T\pi_{0}\left[\exp\left(\frac{T^{1/2}\left|X_{1}\right|}{4}\right)\right]\pi_{0}\left[\left|X_{1}\right|\right] \,.
\end{align*}
As $X_{1}$ is centered and normally distributed of variance $\gamma_{0}$, $\pi_{0}\left[\left|X_{1}\right|\right] = \left(2\gamma_{0}/\pi\right)^{1/2}$ and $\pi_{0}[\exp(T^{1/2}\left|X_{1}\right|/4)] = \gamma_{0}T^{1/2}\exp(\gamma_{0}T/32)/4$.

From $n\geq M^{*}\left(T,\varepsilon\right)=9 \gamma_{0}^{3}T^{2}\exp\left(\gamma_{0}T/16\right)/(2\pi\varepsilon^{2}\log^{3}T)$ the result of Theorem~\ref{theorem:oracle_numerical_Gibbs} is reached. This bound of $M\left(T,\varepsilon\right)$ is prohibitive from a computational viewpoint. That is why we limit the number of iterations to a fixed $n^{*}$.

What we obtain from MCMC is $\bar{f}_{\eta_{T},T,n}\left(\bm y\left|X\right.\right)=\bar{\btheta}'_{\eta_{T},T,n}\left(X\right)\bm y_{1:d_{T}}$ with $\bar{\btheta}_{\eta_{T},T,n}\left(X\right)=\sum_{i=0}^{n-1}\btheta_{\eta_{T},T,i}\left(X\right)/n$. Remark that  $\bar{f}_{\eta_{T},T,n}\left(\cdot\left|X\right.\right)=f_{\bar{\btheta}_{\eta_{T},T,n}\left(X\right)}$. The risk is expressed as 
$$
R\left(\bar{f}_{\eta_{T},T,n}\left(\cdot\left|X\right.\right)\right)=\frac{\left(2\left(\bar{\btheta}_{\eta_{T},T,n}\left(X\right)-\btheta\right)'\Gamma\left(Y\right)\left(\bar{\btheta}_{\eta_{T},T,n}\left(X\right)-\btheta\right)+2\sigma^{2}\right)^{1/2}}{\pi^{1/2}}\;.
$$

\subsection{Numerical work}
Consider $100$ realisations of an autoregressive processes $X$ simulated with the same $\btheta\in s_{d}\left(\delta\right)$ for $d=8$ and $\delta=3/4$ and with $\sigma=1$. Let $\bm c^{(i)}$, $i=1,2$ the sequences defining two different priors in the model order:
\begin{enumerate}
\item $c_{k}^{(1)}=k^{-2}$, the sparsity is favoured,
\item $c_{k}^{(2)}=\rme^{-k}$, the sparsity is strongly favoured.
\end{enumerate}
For each sequence $\bm c$ and for each value of 
$T\in\{2^{j}, j=6,\ldots,12\}$ we compute $\bar{\btheta}_{\eta_{T},T,n^{*}}$, the MCMC approximation of the Gibbs estimator using Algorithm~\ref{algorithm:independent_hastings_sampler} with $\eta=\eta_{T}$.\\
\\
\begin{algorithm}[H] \caption{Independent Hastings Sampler} \label{algorithm:independent_hastings_sampler}
\KwIn{the sample $X_{1:T}$ of $X$, the prior $\bm c$, the learning rate $\eta$, the generators $\pi_{k,T}$ for $k=1,\ldots,d_{T}$ and a maximum iterations number $n^{*}$}
\KwInit{$\btheta_{\eta,T,0}=\bm 0$}
\For{i=1 \KwTo $n^{*}-1$}{
generate $k\in\{1,\ldots,d_{T}\}$ using the prior $\bm c$\;
generate $\btheta_{candidate}\sim \pi_{k,T}$\;
generate $U\sim\mathcal{U}(0,1)$\;
\If {$U\leq \alpha_{\eta,T,X}(\btheta_{\eta,T,i-1},\btheta_{candidate})$}{
    {$\btheta_{\eta,T,i}=\btheta_{candidate}$}
\Else
    {$\btheta_{\eta,T,i}=\btheta_{\eta,T,i-1}$\;}
}
}
\Return{$\bar{\btheta}_{\eta,T,n^{*}}\left(X\right)=\sum_{i=0}^{n^{*}-1}\btheta_{\eta,T,k}\left(X\right)/n^{*}$.}
\end{algorithm}
\vspace{0.3cm}
The acceptance rate is computed as $\alpha_{\eta,T,X}(\btheta_{1},\btheta_{2})=\exp\left(\eta r_{T}\left(\btheta_{1}\left|X\right.\right)-\eta r_{T}\left(\btheta_{2}\left|X\right.\right)\right)$.

Algorithm \ref{algorithm:pi_kT_generation} used by the distributions $\pi_{k,T}$ generates uniform random vectors on $s_{k}\left(1\right)$ by the method described in \cite{Beadle_Djuric:1997}. It relies in the Levinson-Durbin
recursion algorithm. We also implemented the numerical improvements of \cite{Andrieu_Doucet:1999}.

Set $\varepsilon=0.1$. Figure~\ref{figure:quantiles} displays the $(1-\varepsilon)$-quantiles in data $R(\bar{\btheta}_{\eta_{T},T,n^{*}}\left(X\right))-(2/\pi)^{1/2}\sigma^{2}$ for $\bm c^{(1)}$ and $\bm c^{(2)}$ using different values of $n^{*}$.

\begin{figure}[!h]
\begin{minipage}[t]{0.4\linewidth}
\centering
\includegraphics[width=5.7cm]{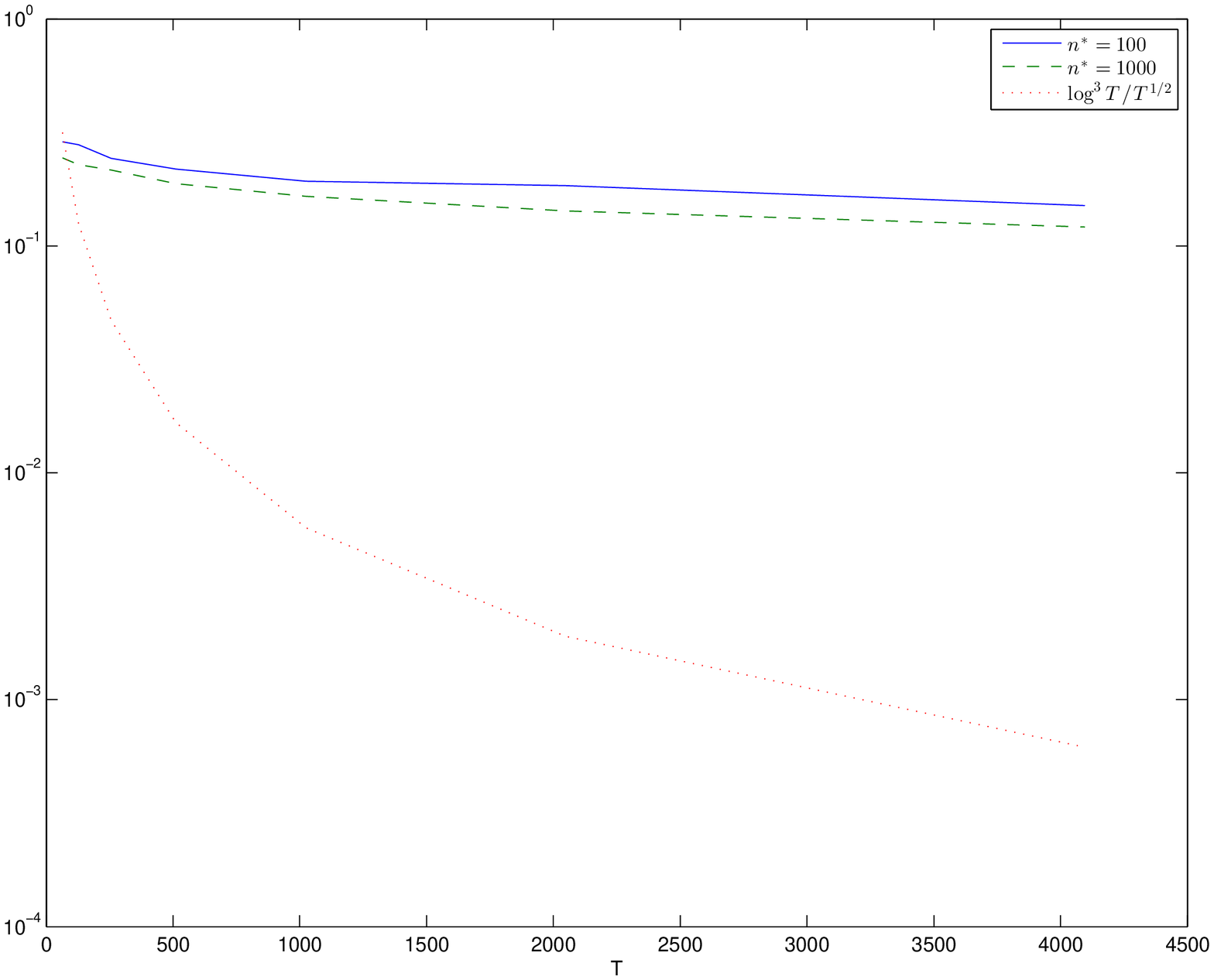}
\end{minipage}
\hspace{1.2cm}
\begin{minipage}[t]{0.4\linewidth}
\centering
\includegraphics[width=5.7cm]{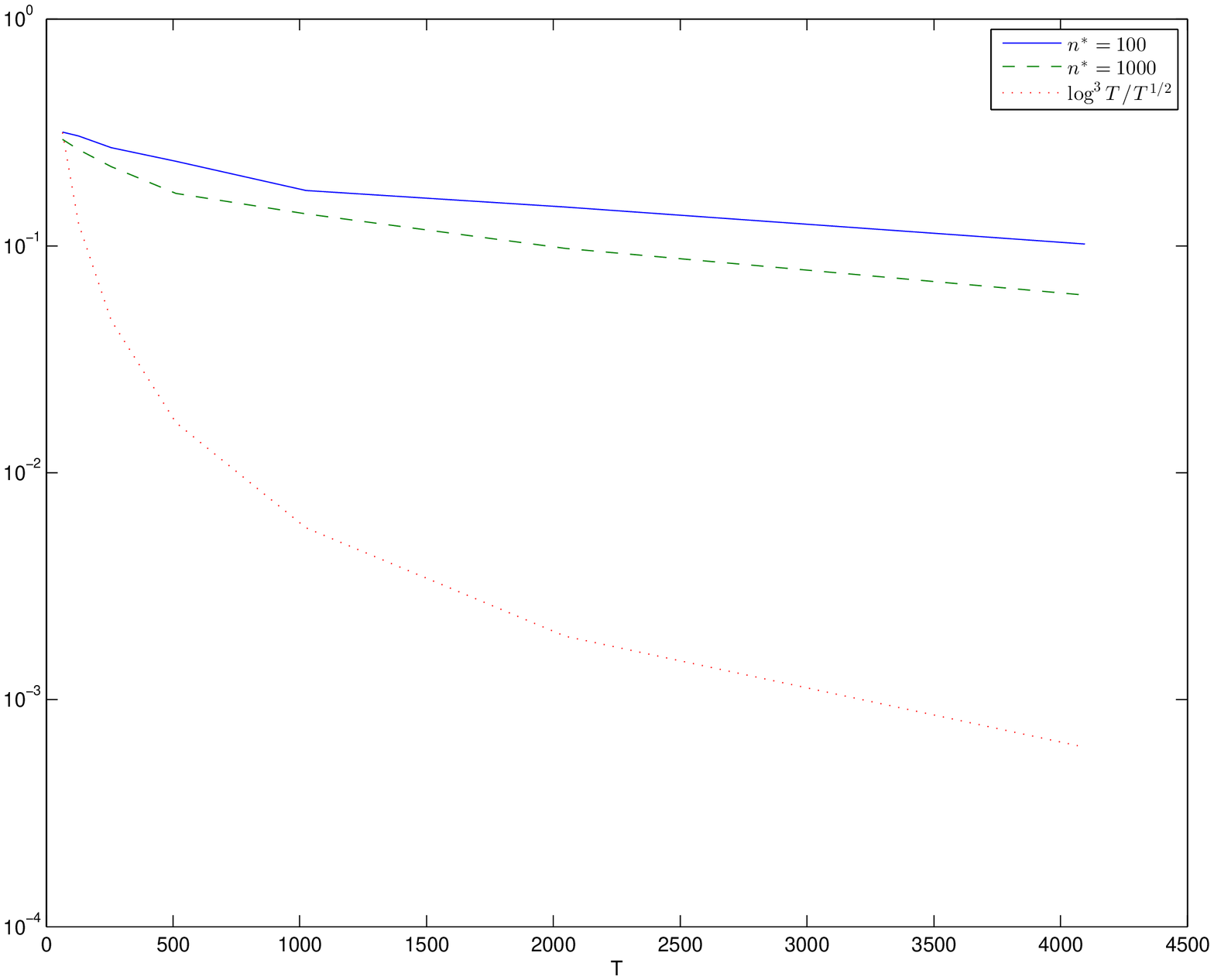}
\end{minipage}
\caption{The plots represent the $0.9$-quantiles in data $R(\bar{\btheta}_{\eta_{T},T,n^{*}}\left(X\right))-(2/\pi)^{1/2}\sigma^{2}$ for $T=32,64,\ldots,4096$. The graph on the left corresponds to the order prior $c_{k}^{(1)}=k^{-2}$ while that on the right corresponds to $c_{k}^{(2)}=\rme^{-k}$. The solid curves were plotted with $n^{*}=100$, the dashed ones with $n^{*}=1000$ and as a reference, the dotted curve is proportional to $\log^{3}T/T^{1/2}$.}
\label{figure:quantiles} 
\end{figure}

Note that, for the proposed algorithm the prediction risk decreases very slowly when the number $T$ of observations grows and the number of MCMC iterations remains constant. If $n^{*}=1000$ the decaying rate is faster than if $n^{*}=100$ for smaller values of $T$. For $T\geq 2000$ we observe that both rates are roughly the same in the logarithmic scale. This behaviour is similar in both cases presented in Figure~\ref{figure:quantiles}. As expected, the risk of the approximated predictor does not converge as $\log^{3}T/T^{1/2}$. 

\section{Discussion} \label{section:discussion}
There are two sources of error in our method: prediction (of the exact Gibbs predictor) and approximation (using the MCMC). The first one decays when $T$ grows and the obtained guarantees for the second one 
explode. We found a possibly pessimistic upper bound for $M(T,\epsilon)$. The exponential growing of this bound is the main weakness of our procedure. The use of a better adapted proposal in the MCMC algorithm needs to be investigated. The Metropolis Langevin Algorithm (see \cite{Atchade:2006}) gives us an insight in this direction. However it is encouraging to see that, in the analysed practical case, the risk of $\bar{f}_{\eta_{T},T,n^{*}}\left(\cdot\left|X\right.\right)$ does not increase with $T$. 

\section*{Acknowledgements}
The author is specially thankful to Fran\c{c}ois Roueff, Christophe Giraud, Peter Weyer-Brown and the two referees for their extremely careful readings and highly pertinent remarks which substantially improved the paper. This work has been partially supported by the Conseil r\'egional
d'\^{I}le-de-France under a doctoral allowance of its program R\'eseau de
Recherche Doctoral en Math\'ematiques de l'\^{I}le de France (RDM-IdF) for the
period 2012 - 2015 and by the Labex LMH (ANR-11-IDEX-003-02).

\section{Technical proofs} \label{section:technical_proofs}

\subsection{Proof of Theorem~\ref{theorem:oracle_Gibbs}} \label{subsection:proof_theorem_oracle_Gibbs}

The proof of Theorem~\ref{theorem:oracle_Gibbs} is based on the same tools used by \cite{Alquier_Wintenberger:2012} up to Lemma~\ref{lemma:exponential_r_T}. For the sake of completeness we quote the essential ones.

We denote by $\mathcal{M}_{+}^{1}\left(F\right)$ the set of probability measures on the measurable space $(F,\mathcal{F})$. Let $\rho,\nu\in \mathcal{M}_{+}^{1}\left(F\right)$, $\K\left(\rho,\nu\right)$ stands for the Kullback-Leibler divergence of $\nu$ from $\rho$.
\begin{eqnarray*}
\K\left(\rho,\nu\right) &=& \left\{
\begin{array}{ll}
\int\log\frac{\rmd\rho}{\rmd\nu}\left(\btheta\right)\rho\left(\rmd\btheta\right) & \textrm{, if } \rho\ll\nu \\
+\infty & \textrm{, otherwise\;.}
\end{array}
\right.
\end{eqnarray*}
The first lemma can be found in \cite[Equation 5.2.1]{Catoni:2004}.
\begin{lemma}[Legendre transform of the Kullback divergence function] \label{lemma:Legendre_KB_divergence} Let $(F,\mathcal{F})$ be any measurable space. For any $\nu\in\mathcal{M}_{+}^{1}\left(F\right)$ and any measurable function $h~:~F~\rightarrow~\R$ such that $\nu\left[\exp\left(h\right)\right]<\infty$ we have,
\begin{eqnarray*}
\nu\left[\exp\left(h\right)\right] &=& \exp\left(\sup\limits_{\rho\in\mathcal{M}_{+}^{1}\left(F\right)}\left(\rho\left[h\right]-\mathcal{K}\left(\rho,\nu\right)\right)\right) \;,
\end{eqnarray*}
with the convention $\infty-\infty=-\infty$. Moreover, as soon as $h$ is upper-bounded on the support of $\nu$, the supremum with respect to $\rho$ in the right-hand side is reached by the Gibbs measure $\nu\left\{h\right\}$. 
\end{lemma}
For a fixed $C>0$, let  $\widetilde{\xi}_{t}^{\left(C\right)}=\max\left\{\min\left\{\xi_{t},C\right\},-C\right\}$. Consider $\widetilde{X}_{t}=H(\widetilde{\xi}_{t}^{\left(C\right)},\widetilde{\xi}_{t-1}^{\left(C\right)},\ldots)$. 

Denote $\widetilde{X}=(\widetilde{X}_{t})_{t\in\Z}$ and by $\widetilde{R}\left(\btheta\right)$ and $\widetilde{r}_{T}\left(\btheta\left|\widetilde{X}\right.\right)$ the respective exact and empirical risks associated with $\widetilde{X}$ in $\btheta$. 
\begin{eqnarray*}
\widetilde{R}\left(\btheta\right) &=& \EE\left[\ell\left(\widehat{\tilde{X}}_{t}^{\btheta},\widetilde{X}_{t}\right)\right] \;, \\
\widetilde{r}_{T}\left(\btheta\left|\widetilde{X}\right.\right) &=& \frac{1}{T-d\left(\btheta\right)}\sum\limits_{t=d\left(\btheta\right)+1}^{T}\ell\left(\widehat{\tilde{X}}_{t}^{\btheta},\widetilde{X}_{t}\right) \;,
\end{eqnarray*}
where $\widehat{\tilde{X}}_{t}^{\btheta} = f_{\btheta}((\widetilde{X}_{t-i})_{i\geq 1})$.

This thresholding is interesting because truncated CBS are weakly dependent processes (see \cite[Section 4.2]{Alquier_Wintenberger:2012}).

A Hoeffding type inequality introduced in \cite[Theorem 1]{Rio:2000} provides useful controls on the difference between empirical and exact risks of a truncated process.

\begin{lemma}[Laplace transform of the risk] \label{lemma:Laplace_risk}
Let $\ell$ be a loss function meeting Assumption~\ref{assumption:Lipschitz_loss} and $X=\left(X_{t}\right)_{t\in\Z}$ a process satisfying Assumption~\ref{hyp:CBS}. For all $T\geq 2$, any $\left\{f_{\btheta},\btheta\in\Theta_{T}\right\}$ satisfying Assumption~\ref{assumption:Lipschitz}, $\Theta_{T}$ such that $d_{T}$, defined in (\ref{equation:d_T}), is at most $T/2$, any truncation level $C>0$, $\eta\geq 0$ and $\btheta\in\Theta_{T}$ we have,
\begin{eqnarray}
\EE\left[\exp\left(\eta\left(\widetilde{R}(\btheta)-\widetilde{r}_{T}\left(\btheta\left|\widetilde{X}\right.\right)\right)\right)\right] &\leq& \exp\left(\frac{4\eta^{2}k^{2}(T,C)}{T}\right) \;, \label{inequality:Laplace_R_r}
\end{eqnarray}
and 
\begin{eqnarray}
\EE\left[\exp\left(\eta\left(\widetilde{r}_{T}\left(\btheta\left|\widetilde{X}\right.\right)-\widetilde{R}(\btheta)\right)\right)\right] &\leq& \exp\left(\frac{4\eta^{2}k^{2}(T,C)}{T}\right) \;, \label{inequality:Laplace_r_R}
\end{eqnarray}
where $k(T,C)=2^{1/2}CK(1+L_{T})\left(A_{*}+\tilde{A}_{*}\right)$. The constants $\tilde{A}_{*}$ and $A_{*}$ are defined in (\ref{equation:tilde_A}) and
(\ref{equation:A}) respectively, $K$ and $L_{T}$ in Assumptions \ref{assumption:Lipschitz_loss} and \ref{assumption:Lipschitz} respectively.
\end{lemma}

The following lemma is a slight modification of \cite[Lemma 6.5]{Alquier_Wintenberger:2012}. It links the two versions of the empirical risk: original and truncated.
\begin{lemma} \label{lemma:exponential_r_T}
Suppose that Assumption~\ref{assumption:Lipschitz_loss} holds for the loss function $\ell$, Assumption~\ref{hyp:CBS} holds for $X=\left(X_{t}\right)_{t\in\Z}$ and Assumption~\ref{hyp:exponential-moment} holds for the innovations with $\zeta=A_{*}$; $A_{*}$ is defined in 
(\ref{equation:A}). For all $T\geq 2$, any $\left\{f_{\btheta},\btheta\in\Theta_{T}\right\}$ meeting Assumption~\ref{assumption:Lipschitz} with $\Theta_{T}$ such that $d_{T}$, defined in (\ref{equation:d_T}), is at most $T/2$, any truncation level $C>0$ and any $0\leq\eta\leq T/4\left(1+L_{T}\right)$ we have,
\begin{eqnarray*}
\EE\left[\exp\left(\eta\sup\limits_{\btheta\in\Theta_{T}}\left|r_{T}\left(\btheta\left|X\right.\right)-\widetilde{r}_{T}\left(\btheta\left|\widetilde{X}\right.\right)\right|\right)\right] &\leq& \exp\left(\eta\varphi\left(T,C,\eta\right)\right)\;,
\end{eqnarray*}
where 
\begin{eqnarray*}
\varphi(T,C,\eta) &=& 2K(1+L_{T})\phi(A_{*})\left(\frac{A_{*}C}{\exp\left(A_{*}C\right)-1}+\eta\frac{4K(1+L_{T})}{T}\right)\;,
\end{eqnarray*}
with $K$ and $L_{T}$ defined in Assumptions \ref{assumption:Lipschitz_loss} and \ref{assumption:Lipschitz} respectively.
\end{lemma}

Finally we present a result on the aggregated predictor defined in (\ref{equation:Gibbs_predictor}). The proof is partially inspired by that of \cite[Theorem 3.2]{Alquier_Wintenberger:2012}.
\begin{lemma} \label{lemma:risk_bound_inf_probability}
Let $\ell$ be a loss function such that Assumption~\ref{assumption:Lipschitz_loss} holds and let $X~=~\left(X_{t}\right)_{t\in\Z}$ a process satisfying Assumption~\ref{hyp:CBS} with probability distribution $\pi_{0}$. For each $T\geq 2$ let $\left\{f_{\btheta},\btheta\in\Theta_{T}\right\}$ be a set of predictors and $\pi_{T}\in\mathcal{M}_{+}^{1}\left(\Theta_{T}\right)$ any prior probability distribution on $\Theta_{T}$. We build the predictor $\hat{f}_{\eta,T}\left(\cdot\left|X\right.\right)$ following (\ref{equation:Gibbs_predictor}) with any $\eta>0$. For any $\varepsilon>0$ and any truncation level $C>0$, with $\pi_{0}$-probability at least $1-\varepsilon$ we have,
\begin{multline*}
R\left(\hat{f}_{\eta,T}\left(\cdot\left|X\right.\right)\right) \leq \inf\limits_{\rho\in\mathcal{M}_{+}^{1}\left(\Theta_{T}\right)}\left\{\rho\left[R\right]+\frac{2\K\left(\rho,\pi_{T}\right)}{\eta}\right\}+\frac{2\log\left(\displaystyle\frac{2}{\varepsilon}\right)}{\eta} \\
+\frac{1}{2\eta}\log\left(\EE\left[\exp\left(2\eta\left(\widetilde{R}-\widetilde{r}_{T}\right)\right)\right]\right)+\frac{1}{2\eta}\log\left(\EE\left[\exp\left(2\eta\left(\widetilde{r}_{T}-\widetilde{R}\right)\right)\right]\right)\\
+\frac{2}{\eta}\log\left(\EE\left[\exp\left(2\eta\sup_{\btheta\in\Theta_{T}}\left|r_{T}\left(\btheta\left|X\right.\right)-\widetilde{r}_{T}\left(\btheta\left|\widetilde{X}\right.\right)\right|\right)\right]\right)\;.
\end{multline*}
\end{lemma}

\begin{proof}
We use Tonelli's theorem and Jensen's inequality with the convex function $g$ to obtain an upper bound for $R\left(\hat{f}_{\eta,T}\left(\cdot\left|X\right.\right)\right)$
\begin{multline*}
R\left(\hat{f}_{\eta,T}\left(\cdot\left|X\right.\right)\right) = \int\limits_{\mathcal{X}^{\Z}}g\left(\int\limits_{\Theta_{T}}\left(f_{\btheta}\left(\left(y_{t-i}\right)_{i\geq 1}\right)-y_{t}\right)\pi_{T}\left\{-\eta r_{T}\left(\cdot\left|X\right.\right)\right\}\left(\rmd\btheta\right)\right)\pi_{0}\left(\rmd\bm y\right)\\
\leq \int\limits_{\mathcal{X}^{\Z}}\left[\int\limits_{\Theta_{T}}g\left(f_{\btheta}\left(\left(y_{t-i}\right)_{i\geq 1}\right)-y_{t}\right)\pi_{T}\left\{-\eta r_{T}\left(\cdot\left|X\right.\right)\right\}\left(\rmd\btheta\right)\right]\pi_{0}\left(\rmd\bm y\right)\\
= \int\limits_{\Theta_{T}}\left[\int\limits_{\mathcal{X}^{\Z}}g\left(f_{\btheta}\left(\left(y_{t-i}\right)_{i\geq 1}\right)-y_{t}\right)\pi_{0}\left(\bm y\right)\right]\pi_{T}\left\{-\eta r_{T}\left(\cdot\left|X\right.\right)\right\}\left(\rmd\btheta\right) = \pi_{T}\left\{-\eta r_{T}\left(\cdot\left|X\right.\right)\right\}\left[R\right] \;.
\end{multline*}

In the remainder of this proof we search for upper bounding $\pi_{T}\left\{-\eta r_{T}\left(\cdot\left|X\right.\right)\right\}\left[R\right]$.

First, we use the relationship:
\begin{eqnarray}
R-r_{T}\left(\cdot\left|X\right.\right)=\left(\widetilde{R}-\widetilde{r}_{T}\left(\cdot\left|\widetilde{X}\right.\right)\right)+\left(R-\widetilde{R}\right)-\left(r_{T}\left(\cdot\left|X\right.\right)-\widetilde{r}_{T}\left(\cdot\left|\widetilde{X}\right.\right)\right)\;. \label{equation:R_r_T}
\end{eqnarray}

For the sake of simplicity and while it does not disrupt the clarity, we lighten the notation of $r_{T}$ and $\widetilde{r}_{T}$. We now suppose that in the place of $\btheta$ we have a random variable distributed as $\pi_{T}\in\mathcal{M}_{+}^{1}\left(\Theta_{T}\right)$. This is taken into account in the following expectations. The identity (\ref{equation:R_r_T}) and the Cauchy-Schwarz inequality lead to

\begin{multline} 
\EE\left[\exp\left(\frac{\eta}{2}\left(R-r_{T}\right)\right)\right] = \EE\left[\exp\left(\frac{\eta}{2}\left(\widetilde{R}-\widetilde{r}_{T}\right)\right)\exp\left(\frac{\eta}{2}\left(\left(R-\widetilde{R}\right)-\left(r_{T}-\widetilde{r}_{T}\right)\right)\right)\right] \\
\leq \left(\EE\left[\exp\left(\eta\left(\widetilde{R}-\widetilde{r}_{T}\right)\right)\right]\EE\left[\exp\left(\eta\left(\left(R-\widetilde{R}\right)-\left(r_{T}-\widetilde{r}_{T}\right)\right)\right)\right]\right)^{1/2} \\
\leq \left(\EE\left[\exp\left(\eta\left(\widetilde{R}-\widetilde{r}_{T}\right)\right)\right]\EE\left[\exp\left(\eta\sup\limits_{\btheta\in\Theta_{T}}\left|\left(R-\widetilde{R}\right)\left(\btheta\right)-\left(r_{T}-\widetilde{r}_{T}\right)\left(\btheta\right)\right|\right)\right]\right)^{1/2} \label{inequality:R_r_n_Cauchy} \;. 
\end{multline}
Observe now that $R\left(\btheta\right)=\EE\left[r_{T}\left(\btheta\left|X\right.\right)\right]$ and $\widetilde{R}\left(\btheta\right)=\EE[\widetilde{r}_{T}(\btheta|\widetilde{X})]$. Jensen's inequality for the exponential function gives that 
\begin{align}
\exp\left(\eta\sup\limits_{\btheta\in\Theta_{T}}\left|R\left(\btheta\right)-\widetilde{R}\left(\btheta\right)\right|\right) &\leq \exp\left(\eta\EE\left[\sup\limits_{\btheta\in\Theta_{T}}\left|r_{T}\left(\btheta\left|X\right.\right)-\widetilde{r}_{T}\left(\btheta\left|\widetilde{X}\right.\right)\right|\right]\right) \notag\\
&\leq \EE\left[\exp\left(\eta\sup\limits_{\btheta\in\Theta_{T}}\left|r_{T}\left(\btheta\left|X\right.\right)-\widetilde{r}_{T}\left(\btheta\left|\widetilde{X}\right.\right)\right|\right)\right] \;.\label{inequality_R_tilde_R}
\end{align}

From (\ref{inequality_R_tilde_R}) we see that 
\begin{multline}
\EE\left[\exp\left(\eta\sup\limits_{\btheta\in\Theta_{T}}\left|\left(R-\widetilde{R}\right)\left(\btheta\right)-\left(r_{T}-\widetilde{r}_{T}\right)\left(\btheta\right)\right|\right)\right] \\
\leq \EE\left[\exp\left(\eta\sup\limits_{\btheta\in\Theta_{T}}\left|R\left(\btheta\right)-\widetilde{R}\left(\btheta\right)\right|\right)\exp\left(\eta\sup\limits_{\btheta\in\Theta_{T}}\left|r_{T}\left(\btheta\left|X\right.\right)-\widetilde{r}_{T}\left(\btheta\left|\widetilde{X}\right.\right)\right|\right)\right] \\
\leq \left(\EE\left[\exp\left(\eta\sup\limits_{\btheta\in\Theta_{T}}\left|r_{T}\left(\btheta\left|X\right.\right)-\widetilde{r}_{T}\left(\btheta\left|\widetilde{X}\right.\right)\right|\right)\right]\right)^{2} \;. \label{inequality:quotient}
\end{multline}

Combining (\ref{inequality:R_r_n_Cauchy}) and (\ref{inequality:quotient}) we obtain 
\begin{multline}
\EE\left[\exp\left(\frac{\eta}{2}\left(R-r_{T}\left(\cdot\left|X\right.\right)\right)\right)\right] \leq \left(\EE\left[\exp\left(\eta\left(\widetilde{R}-\widetilde{r}_{T}\right)\right)\right]\right)^{1/2}\\
\EE\left[\exp\left(\eta\sup\limits_{\btheta\in\Theta_{T}}\left|r_{T}\left(\btheta\left|X\right.\right)-\widetilde{r}_{T}\left(\btheta\left|\widetilde{X}\right.\right)\right|\right)\right] \label{equation:bound_R_r_T} \;. 
\end{multline}

Let $L_{\eta,T,C}=\log((\EE[\exp(\eta(\widetilde{R}-\widetilde{r}_{T}))])^{1/2}\EE[\exp(\eta\sup_{\btheta\in\Theta_{T}}|r_{T}(\btheta|X)-\widetilde{r}_{T}(\btheta|\widetilde{X})|)])$. Remark that the left term of (\ref{equation:bound_R_r_T}) is equal to the integral of the expression enclosed in brackets with respect to the measure $\pi_{0}\times\pi_{T}$. Changing $\eta$ by $2\eta$ and thanks to Lemma~\ref{lemma:Legendre_KB_divergence} we get
\begin{eqnarray*}
\pi_{0}\left[\exp\left(\sup\limits_{\rho\in\mathcal{M}_{+}^{1}\left(\Theta_{T}\right)}\left(\eta\rho[R-r_{T}\left(\cdot\left|X\right.\right)]-\K\left(\rho,\pi_{T}\right)\right)\right)\right] \leq \exp\left(L_{2\eta,T,C}\right) \;.
\end{eqnarray*}

Markov's inequality implies that for all $\varepsilon>0$, with $\pi_{0}$- probability at least $1-\varepsilon$
\begin{eqnarray*}
\sup\limits_{\rho\in\mathcal{M}_{+}^{1}\left(\Theta_{T}\right)}\left(\eta\rho\left[R-r_{T}\left(\cdot\left|X\right.\right)\right]-\K\left(\rho,\pi_{T}\right)\right)-\log\left(\frac{1}{\varepsilon}\right)-L_{2\eta,T,C} \leq 0 \;.
\end{eqnarray*}

Hence, for any $\pi_{T}\in\mathcal{M}_{+}^{1}\left(\Theta_{T}\right)$ and $\eta>0$, with $\pi_{0}$- probability at least $1-\varepsilon$, for all $\rho\in\mathcal{M}_{+}^{1}\left(\Theta_{T}\right)$
\begin{eqnarray}
\rho\left[R-r_{T}\left(\cdot\left|X\right.\right)\right]-\frac{1}{\eta}\K\left(\rho,\pi_{T}\right)-\frac{1}{\eta}\log\left(\frac{1}{\varepsilon}\right)-\frac{L_{2\eta,T,C}}{\eta} &\leq& 0\;. \label{inequality:R_r_T}
\end{eqnarray}
By setting $\rho=\pi_{T}\{-\eta r_{T}\left(\cdot\left|X\right.\right)\}$ and relying on Lemma~\ref{lemma:Legendre_KB_divergence}, we have
\begin{eqnarray*}
\K\left(\pi_{T}\left\{-\eta r_{T}\right\},\pi_{T}\right) &=& \pi_{T}\left\{-\eta r_{T}\right\}\left[\log\frac{\rmd\pi_{T}\left\{-\eta r_{T}\right\}}{\rmd\pi_{T}}\right] = \pi_{T}\left\{-\eta r_{T}\right\}\left[\log\frac{\exp\left(-\eta r_{T}\right)}{\pi_{T}\left[\exp\left(-\eta r_{T}\right)\right]}\right] \\
&=& \pi_{T}\left\{-\eta r_{T}\right\}\left[-\eta r_{T}\right]-\log\left(\pi_{T}\left[\exp\left(-\eta r_{T}\right)\right]\right) \\
&=& \pi_{T}\left\{-\eta r_{T}\right\}\left[-\eta r_{T}\right]+\inf\limits_{\rho\in\mathcal{M}_{+}^{1}\left(\Theta_{T}\right)}\left\{\rho\left[\eta r_{T}\right]+\K\left(\rho,\pi_{T}\right)\right\} \\
\end{eqnarray*}
Using (\ref{inequality:R_r_T}) with $\rho=\pi_{T}\{-\eta r_{T}\left(\cdot\left|X\right.\right)\}$ it follows that, with $\pi_{0}$- probability at least $1-\varepsilon$,
\begin{align*}
\pi_{T}\left\{-\eta r_{T}\left(\cdot\left|X\right.\right)\right\}\left[R\right] \leq& \inf\limits_{\rho\in\mathcal{M}_{+}^{1}\left(\Theta_{T}\right)}\left\{\rho\left[r_{T}\left(\cdot\left|X\right.\right)\right]+\frac{\K\left(\rho,\pi_{T}\right)}{\eta}\right\}+\frac{\log\left(\displaystyle\frac{1}{\varepsilon}\right)}{\eta}+\frac{L_{2\eta,T,C}}{\eta}\;.
\end{align*}

To upper bound $\rho[r_{T}(\cdot|X)]$ we use an upper bond on $\rho\left[r_{T}(\cdot|X)-R\right]$. We obtain an inequality similar to (\ref{inequality:R_r_T}) with $\rho\left[R-r_{T}(\cdot|X)\right]$ replaced by $\rho\left[r_{T}(\cdot|X)-R\right]$ and $L_{\eta,T,C}$ replaced by $L'_{\eta,T,C}=\log((\EE[\exp(\eta(\widetilde{r}_{T}-\widetilde{R}))])^{1/2}\EE[\exp(\eta\sup_{\btheta\in\Theta_{T}}|r_{T}(\btheta|X)-\widetilde{r}_{T}(\btheta|\widetilde{X})|)])$. This provides us another inequality satisfied with $\pi_{0}$- probability at least $1-\varepsilon$. To obtain a $\pi_{0}$- probability of the intersection larger than $1-\varepsilon$ we apply previous computations with $\varepsilon/2$ instead of $\varepsilon$ and hence,
\begin{multline*}
\pi_{T}\left\{-\eta r_{T}\left(\cdot\left|X\right.\right)\right\}\left[R\right] \leq \inf\limits_{\rho\in\mathcal{M}_{+}^{1}\left(\Theta_{T}\right)}\left\{\rho\left[R\right]+\frac{2\K\left(\rho,\pi_{T}\right)}{\eta}\right\}+\frac{2\log\left(\displaystyle\frac{2}{\varepsilon}\right)}{\eta}\\
+\frac{1}{2\eta}\log\left(\EE\left[\exp\left(2\eta\left(\widetilde{R}-\widetilde{r}_{T}\right)\right)\right]\right)+\frac{1}{2\eta}\log\left(\EE\left[\exp\left(2\eta\left(\widetilde{r}_{T}-\widetilde{R}\right)\right)\right]\right)\\
+\frac{2}{\eta}\log\left(\EE\left[\exp\left(2\eta\sup_{\btheta\in\Theta_{T}}\left|r_{T}\left(\btheta\left|X\right.\right)-\widetilde{r}_{T}\left(\btheta\left|\widetilde{X}\right.\right)\right|\right)\right]\right)\;.
\end{multline*} 
\end{proof}

We can now proof Theorem~\ref{theorem:oracle_Gibbs}.
\begin{proof}
Let $\pi_{0,C}$ denote the distribution on $\mathcal{X}^{\Z}\times\mathcal{X}^{\Z}$ of the couple $(X,\widetilde{X})$. Fubini's theorem and (\ref{inequality:Laplace_R_r}) of Lemma~\ref{lemma:Laplace_risk} imply that 
\begin{multline}
\EE\left[\exp\left(2\eta\left(\widetilde{R}-\widetilde{r}_{T}\right)\right)\right] = \pi_{0,C}\times\pi_{T}\left[\exp\left(2\eta\left(\widetilde{R}-\widetilde{r}_{T}\right)\right)\right] = \pi_{T}\times\pi_{0,C}\left[\exp\left(2\eta\left(\widetilde{R}-\widetilde{r}_{T}\right)\right)\right] \\
\leq \exp\left(\frac{16\eta^{2}k^{2}(T,C)}{T}\right) \;. \label{inequality:R_r_risk_Kullback}
\end{multline}
Using (\ref{inequality:Laplace_r_R}), we analogously get
\begin{eqnarray}
\EE\left[\exp\left(2\eta\left(\widetilde{r}_{T}-\widetilde{R}\right)\right)\right] \leq \exp\left(\frac{16\eta^{2}k^{2}(T,C)}{T}\right)\;. \label{inequality:r_R_risk_Kullback}
\end{eqnarray}
Consider the set of probability measures $\left\{\rho_{\btheta_{T},\Delta}, T\geq 2,0\leq\Delta\leq \Delta_{T}\right\}\subset \mathcal{M}_{+}^{1}\left(\Theta_{T}\right)$, where $\btheta_{T}$ is the parameter defined by Assumption~\ref{assumption:theta_balls} and $\rho_{\btheta_{T},\Delta}\left(\btheta\right)\propto \pi_{T}\left(\btheta\right)\1_{B\left(\btheta_{T},\Delta\right)\cap\Theta_{T}}\left(\btheta\right)$. Lemma~\ref{lemma:risk_bound_inf_probability}, together with Lemma~\ref{lemma:exponential_r_T}, (\ref{inequality:R_r_risk_Kullback}) and (\ref{inequality:r_R_risk_Kullback}) guarantee that for all $0<\eta\leq T/8\left(1+L_{T}\right)$
\begin{multline}
R\left(\hat{f}_{\eta,T}\left(\cdot\left|X\right.\right)\right) \leq \inf\limits_{0\leq\Delta\leq \Delta_{T}}\left\{\rho_{\btheta_{T},\Delta}\left[R\right]+\frac{2\K\left(\rho_{\btheta_{T},\Delta},\pi_{T}\right)}{\eta}\right\}+\frac{16\eta k^{2}(T,C)}{T}+\frac{2\log\left(\displaystyle\frac{2}{\varepsilon}\right)}{\eta}+\\
 4\varphi(T,C,2\eta) \;. \label{equation:risk_ball}
\end{multline}
Thanks to assumptions \ref{assumption:Lipschitz_loss} and \ref{assumption:theta_Lipschitz}, for any $T\geq 2$ and $\btheta\in B\left(\btheta_{T},\Delta\right)$ 
\begin{eqnarray}
R\left(\btheta\right)-R\left(\btheta_{T}\right) \leq K\pi_{0}\left[\left|\left|f_{\btheta}\left(\left(Y_{t-i}\right)_{i\geq 1}\right)-f_{\btheta_{T}}\left(\left(Y_{t-i}\right)_{i\geq 1}\right)\right|\right|\right] \leq K\mathcal{D}d_{T}^{1/2}\Delta\;. \label{equation:bound_difference_R}
\end{eqnarray}
For $T\geq 4$ Assumption~\ref{assumption:theta_balls} gives
\begin{eqnarray}
\K\left(\rho_{\btheta_{T},\Delta},\pi_{T}\right) = \log\left(\frac{1}{\pi_{T}\left[B\left(\btheta_{T},\Delta\right)\cap\Theta_{T}\right]}\right) \leq -n_{T}^{1/\gamma}\log\left(\Delta\right)-\log\left(\mathcal{C}_{2}\right) \;. \label{equation:bound_Kullback}
\end{eqnarray}
Plugging (\ref{equation:bound_difference_R}) and (\ref{equation:bound_Kullback}) into (\ref{equation:risk_ball}) and using again Assumption~\ref{assumption:theta_balls}
\begin{align}
R\left(\hat{f}_{\eta,T}\left(\cdot\left|X\right.\right)\right) \leq& R\left(\btheta_{T}\right)+\inf\limits_{0\leq\Delta\leq \Delta_{T}}\left\{\mathcal{E}_{1}d_{T}^{1/2}\Delta-\frac{2n_{T}^{1/\gamma}\log\left(\Delta\right)}{\eta}\right\}  +\frac{\mathcal{E}_{2}\eta\left(1+L_{T}\right)^{2} C^{2}}{T}\nonumber \\
&  +\frac{\mathcal{E}_{3}\left(1+L_{T}\right)C}{\exp\left(A_{*}C\right)-1}+\frac{2\log\left(\displaystyle\frac{2}{\varepsilon}\right)-2\log\left(\mathcal{C}_{2}\right)}{\eta}+\frac{\mathcal{E}_{4}\left(1+L_{T}\right)^{2}\eta}{T} \label{equation:bound_risk_general_Gibbs}
\end{align}
where $\mathcal{E}_{1}=K\mathcal{D}$, $\mathcal{E}_{2}=32K^{2}\left(A_{*}+\tilde{A}_{*}\right)^{2}$, $\mathcal{E}_{3}=8K\phi(A_{*})A_{*}$ and $\mathcal{E}_{4}=32K^{2}\phi(A_{*})$.

We upper bound $d_{T}$ by $T/2$, $n_{T}$ by $\log^{\gamma}T$ and substitute $\Delta_{T}=\mathcal{C}_{3}/T$. Since it is difficult to minimize the right term of (\ref{equation:bound_risk_general_Gibbs}) with respect to $\eta$ and $C$ at the same time, we evaluate them in certain values to obtain a convenient upper bound. 

At a fixed $\varepsilon$, the convergence rate of $\left[2\log\left(2/\varepsilon\right)-2\log\left(\mathcal{C}_{2}\right)\right]/\eta+\mathcal{E}_{4}\left(1+L_{T}\right)^{2}\eta/T$ is at best $\log T/T^{1/2}$, and we get it doing $\eta\propto T^{1/2}/\log T$. As $\eta\leq T/8(1+L_{T})$ we set $\eta=\eta_{T}=T^{1/2}/(4\log T)$.

The order of the already chosen terms is $\log^{3}T/T^{1/2}$, doing $C=\log T/A_{*}$ we preserve it. 
Taking into account that $R\left(\btheta_{T}\right)\leq \inf_{\btheta\in\Theta_{T}}R\left(\btheta\right)+\mathcal{C}_{1}\log^{3}T/T^{1/2}$ the result follows.
\end{proof}

\subsection{Proof of Proposition~\ref{proposition:numerical_oracle_Gibbs}} \label{subsection:proof_proposition_numerical_oracle_Gibbs}

Considering that Assumption~\ref{assumption:Lipschitz_loss} holds we get 
\begin{align*}
\left|R\left(\bar{f}_{\eta,T,n}\left(\cdot\left|X\right.\right)\right)-R\left(\hat{f}_{\eta,T}\left(\cdot\left|X\right.\right)\right)\right| 
\leq K\int\limits_{\mathcal{X}^{\Z}}\left|\bar{f}_{\eta,T,n}\left(\bm y\left|X\right.\right)-\hat{f}_{\eta,T}\left(\bm y\left|X\right.\right)\right|\pi_{0}\left(\rmd\bm y\right) 
\end{align*}

Observe that the last expression depends on $X_{1:T}$ and $\Phi_{\eta,T}\left(X\right)$. We bound the expectation to infer a bound in probability. 

Tonelli's theorem and Jensen's inequality lead to 
\begin{multline}
\nu_{\eta,T}\left[\left|R\left(\bar{f}_{\eta,T,n}\left(\cdot\left|X\right.\right)\right)-R\left(\hat{f}_{\eta,T}\left(\cdot\left|X\right.\right)\right)\right|\right] \leq \\
K\int\limits_{\mathcal{X}^{\Z}}\int\limits_{\mathcal{X}^{\Z}}\left(\int\limits_{\Theta_{T}^{\N}}\left|\bar{f}_{\eta,T,n}\left(\bm y\left|\bm x\right.\right)-\hat{f}_{\eta,T}\left(\bm y\left|\bm x\right.\right)\right|^{2}\mu_{\eta,T}\left(\rmd\bm\phi\left|\bm x\right.\right)\right)^{1/2}\pi_{0}\left(\rmd\bm y\right)\pi_{0}\left(\rmd\bm x\right)\;. \label{inequality:expectation_risk_mumerical_theoretical}
\end{multline}

We are then interested in upper bounding the expression under the square root.
To that end, we use \cite[Theorem 3.1]{Latuszynski_Niemiro:2011} which implies that for any $\bm x$
\begin{multline*}
\int\limits_{\Theta_{T}^{\N}}\left|\bar{f}_{\eta,T,n}\left(\bm y\left|\bm x\right.\right)-\hat{f}_{\eta,T}\left(\bm y\left|\bm x\right.\right)\right|^{2}\mu_{\eta,T}\left(\rmd\bm\phi\left|\bm x\right.\right) \leq \\ \sup\limits_{\btheta\in\Theta_{T}}\left(f_{\btheta}\left(\bm y\right)-\hat{f}_{\eta,T}\left(\bm y\left|\bm x\right.\right)\right)^{2}\left(\frac{4}{\beta_{\eta,T}\left(\bm x\right)}-3\right)\left(\frac{1}{n}+\frac{2}{n^{2}\beta_{\eta,T}\left(\bm x\right)}\right)\;.
\end{multline*}
Plugging this on ($\ref{inequality:expectation_risk_mumerical_theoretical}$), using that $n\geq 1$ and that
$$
\left(\left(4-3\beta_{\eta,T}\left(\bm x\right)\right)\left(2+\beta_{\eta,T}\left(\bm x\right)\right)\right)^{1/2}\leq 3\;,
$$
we obtain the following 
\begin{multline*}
\nu_{\eta,T}\left[\left|R\left(\bar{f}_{\eta,T,n}\left(\cdot\left|X\right.\right)\right)-R\left(\hat{f}_{\eta,T}\left(\cdot\left|X\right.\right)\right)\right|\right] \leq \\
\frac{3K}{n^{1/2}}\int\limits_{\mathcal{X}^{\Z}}\frac{1}{\beta_{\eta,T}\left(\bm x\right)}\int\limits_{\mathcal{X}^{\Z}}\sup\limits_{\btheta\in\Theta_{T}}\left|f_{\btheta}\left(\bm y\right)-\hat{f}_{\eta,T}\left(\bm y\left|\bm x\right.\right)\right|\pi_{0}\left(\rmd\bm y\right)\pi_{0}\left(\rmd\bm x\right) \,.
\end{multline*}
The result follows from Markov's inequality.

\bibliographystyle{plain}
\bibliography{allbib}

\end{document}